\documentclass[11pt,reqno]{amsart}

\usepackage{amsmath,amsfonts,amssymb,amscd,verbatim,delarray,verbatim}

\DeclareMathSymbol{\twoheadrightarrow}  {\mathrel}{AMSa}{"10}

\def\Q{{\mathbb Q}}
\def\Z{{\mathbb Z}}
                             \def\NN{{\mathbb N}}
\def\C{{\mathbb C}}
\def\wt{\mathbf{wt}}

\def\RR{{\mathbb R}}
\def\F{{\mathbb F}}
\def\P{{\mathbb P}}
             \def\PP{{\mathfrak P}}

             \def\Fr{\mathrm{Fr}}

\def\f{{\tilde F}}
                     \def\f0{{\mathfrak f}}

\def\A8{{\mathbf A}_8}

\def\Gal{\mathrm{Gal}}

                              \def\ord{\mathrm{ord}}
                               \def\mult{\mathrm{mult}}

\def\End{\mathrm{End}}
\def\Aut{\mathrm{Aut}}

\def\Hom{\mathrm{Hom}}

                                          \def\tate{\mathrm{tate}}

\def\I{\mathrm{Id}}

\def\fchar{\mathrm{char}}
\def\GL{\mathrm{GL}}

\def\A{\mathbf{A}}
\def\T{{\mathcal T}}
\def\B{{\mathfrak B}}

\def\dim{\mathrm{dim}}
                           \def\rk{\mathrm{rk}}
                           
                           \def\Sl{\mathrm{Slp}}
\def\Oc{{\mathcal O}}

\newtheorem{thm}{Theorem}[section]
\newtheorem{lem}[thm]{Lemma}
\newtheorem{cor}[thm]{Corollary}
\newtheorem{prop}[thm]{Proposition}
\theoremstyle{definition}
\newtheorem{defn}[thm]{Definition}

\newtheorem{rem}[thm]{Remark}

\newtheorem{rems}[thm]{Remarks}
        \newtheorem{sect}[thm]{}
\title[Tate classes  on Abelian varieties]{Tate classes on self-products of Abelian varieties over finite fields}

\author[Yuri\ G.\ Zarhin]{Yuri\ G.\ Zarhin}
\address{Department of Mathematics, Pennsylvania State University,
University Park, PA 16802, USA}

\thanks{The author  was partially supported by Simons Foundation Collaboration grant   \# 585711.}
\email{zarhin\char`\@math.psu.edu}
\thanks{}
\begin{document}
\begin{abstract}
We deal with $g$-dimensional abelian varieties $X$ over finite fields. We prove that there is an universal constant (positive integer) $N=N(g)$ that depends only on $g$ that enjoys the following property. If  a certain self-product of $X$ carries an {\sl exotic} Tate class then the self-product  $X^{2N}$ of $X$  also carries an exotic Tate class. 
This gives a positive answer to a question of Kiran Kedlaya.
\end{abstract}
\maketitle

\section{Introduction}
Let $X$ be  an abelian  variety of positive dimension $g$ over  a finite field $k=\F_q$ of characteristic $p$ (where $q$ is a power of  $p$),  $\Fr_X$ the
 Frobenius endomorphism of  $X$, and $\mathcal{P}_X[t]\in \Z[t]$ the characteristic  polynomial of $\Fr_X$, which is a degree $2g$ monic polynomial with integer coefficients.
 \cite{Mumford,Tate1}. Let $L=L_X$ be the splitting field of  $\mathcal{P}_X[t]$ over the field $\Q$ of rational numbers and therefore is a number field. Since $\deg(\mathcal{P}_X)=2g$,  the degree $[L_X:\Q]$ divides $(2g)!$. (In fact, one may prove that
  $[L_X:\Q]$  divides  $2^g g!$, see below). We write $R_X$ for
  the set of  eigenvalues of $\Fr_X$; clearly,  $R_X$ coincides with the set of roots of $\mathcal{P}_X[t]\in \Z[t]$ and is viewed as a certain finite subset of  $L_X^{*}$.
  Clearly, $R_X$ consists of algebraic  integers and $\#(R_X) \le 2g$. (The equality holds if and only if $\mathcal{P}_X[t]$ has no repeated roots.) By a classical theorem of A.Weil \cite{Mumford},
   all algebraic numbers $\alpha\in R_X$ have the same archimedean value $\sqrt{q}$. 
 In addition, $\alpha \mapsto q/\alpha$ is a {\sl permutation} of  $R_X$. If $\alpha$ is a root of $\mathcal{P}_X[t]$ (i.e., $\alpha\in R_X$) then we write $\mult_X(\alpha)$ for its multiplicity. It is well known that if $\alpha\in R_X$ then
 \begin{equation}
 \label{multR}
   \mult_X(\alpha)=\mult_X(q/\alpha); \ \text{ if } \alpha=q/\alpha \ \text{ then } \mult_X(\alpha) \ \text{ is even}.
   \end{equation}
   In particular, the constant term $\prod_{\alpha\in R_X}\alpha^{\mult_X(\alpha)}$ of $\mathcal{P}_X[t]$ is $q^g$.
   
   The Galois group $\Gal(L_X/\Q)$ of $L_X/\Q$ permutes elements of $R_X$ and
   
   \begin{equation}
   \label{sigmaRX}
   \mult_X(\sigma(\alpha))=\mult_X(\alpha) \ \forall \sigma\in \Gal(L_X/\Q) , \alpha\in R_X.
   \end{equation}
   
   In this paper we continue our  study of multiplicative relations between elements of $R_X$ that was started in \cite{ZarhinK3,LenstraZarhin,ZarhinEssen,ZarhinJPAA}.
   (In \cite{ZarhinK3,LenstraZarhin} we concentrated on abelian varieties with rather special type of Newton polygons; in \cite{ZarhinEssen,ZarhinJPAA} we studied
   abelian varieties of small dimension).
   In order to state  results of the present paper, we need the following definitions.
   
   \begin{defn}
   An integer-valued function $e \colon R_X \to \Z$ is called
   \begin{itemize}
   \item[(i)]
   {\bf admissible} if there exists an integer $d$ such that
   \begin{equation}
  \label{relationM}
  \prod_{\alpha\in R_X}\alpha^{e(\alpha)}=q^d;
  \end{equation}
  Such a $d$ is called the {\bf degree} of $e$ and is denoted $\deg(e)$. 
  
  The nonnegative integer
  $\sum_{\alpha\in R_X}|e(\alpha)|$  is called the {\bf weight} of $e$ and denoted $\wt(e)$.
   \item[(ii)]
   {\bf trivial}  if
  \begin{equation}
  \label{relationT}
  e(\alpha)=e(q/\alpha) \ \forall \alpha\in R_X;  \ e(\beta)\in 2\Z \ \forall   \beta\in R_X  \text{ with } \beta^2=q.
  \end{equation}
   \end{itemize}
    \end{defn}
    
    \begin{defn}
   An {\sl admissible} integer-valued function $e \colon R_X \to \Z$ is called
   {\bf reduced} if it enjoys the following properties:
  \begin{itemize}
  \item[(i)]
  $\deg(e) \ge 1$ and all $e(\alpha)\ge 0$.
   \item[(ii)]
  If $\alpha \in R_X$ and $\alpha \ne q/\alpha$  then either $e(\alpha)=0$ or $e(q/\alpha)=0$.
   \item[(iii)]
   $e(\beta)=0$ or $1$  $\forall \beta\in R_X$ with $\beta^2=q$.
  \end{itemize}
   \end{defn}

 \begin{rems}
 \begin{itemize}
 \item[]
 \item[(i)] It follows from \eqref{multR} that  $\alpha \mapsto \mult_X(\alpha)$ is a trivial admissible function of degree $g$ and weight $2g$.
 \item[(ii)] Every trivial function is admissible.
 \item[(iii)] If $e \colon R_X \to \Z$ is admissible then it follows from Weil's theorem that
 \begin{equation}
 \label{degreeE}
 2\deg(e)=\sum_{\alpha\in R_X}e(\alpha).
 \end{equation}
 \item[(iv)] If $e \colon R_X \to \Z$ is reduced admissible then it follows from \eqref{degreeE} that
 \begin{equation}
 \label{weightE}
 \wt(e)=2\deg(e).
 \end{equation}
 \end{itemize}
 \end{rems}
  
Our first main result is the following assertion.

\begin{thm}
\label{mainRelation}
 Let $g$ be a positive integer.
There exists a positive integer $N=N(g)$ that depends only on $g$ and enjoys the following property.

Let $X$ be a $g$-dimensional abelian variety  over a finite field $k$  such that there exists a nontrivial
admissible function $R_X \to \Z$.  Then there exists a reduced  admissible function of degree $\le N(g)$.

\end{thm}

Our main tool in the proof of Theorem \ref{mainRelation} is the multiplicative (sub)group $\Gamma(X,k)\subset L_X^{*}$ generated by  $R_X$, which was first introduced in  \cite{ZarhinIzv79,ZarhinInv79} (see also \cite{ZarhinK3,LenstraZarhin,ZarhinEssen,ZarhinJPAA}).

\begin{defn}
\begin{itemize}
\item[]
\item[(i)]
We  say that $k=\F_q$ is {\bf small} with respect to $X$ if there exist  {\sl distinct } $\alpha_1, \alpha_2 \in R_X$ such that
$\alpha_1/\alpha_2$ is a {\sl root of unity}.
\item[(ii)]
We say that  $k$ is {\bf sufficiently large} with respect to $X$ if $\Gamma(X,k)$ does {\sl not} contain  roots of unity except $1$ (see \cite{ZarhinEssen,ZarhinJPAA}).
\end{itemize}
\end{defn}

\begin{rems}
\label{smallN}
\begin{itemize}
\item[]
\item[(i)] If $k$ is {\sl not} small with respect to $X$ then there is at most one $\beta\in R_X$ with $\beta^2=q$.

\item[(ii)] If $k$ is sufficiently large with respect to $X$ then it is {\sl not} small.
\end{itemize}
\end{rems}

The role of  $\Gamma(X,k)$ is explained by the following statement.

\begin{lem}
\label{rankGamma}
Suppose that $k$ is not small with respect to $X$. Then the following three conditions are equivalent.
\begin{itemize}
\item[(i)] There exists a nontrivial admissible function $R_X \to \Z$.
\item[(ii)] There exists a reduced admissible function $R_X \to \Z$.
\item[(iii)] The rank of $\Gamma(X,k)$ does not exceed $\lfloor\#(R_X)/2\rfloor$.
\end{itemize}
\end{lem}

Our second main result deals with Tate classes on abelian varieties (see \cite{Tate0,Tate1,Tate94,ZarhinTokyo,ZarhinEssen} and Section \ref{etaleT} below for the definition of these classes
and their basic properties).
Recall that a Tate  class is called {\sl exotic} if it {\sl cannot} be presented as a linear combination of products of divisor classes.

\begin{thm}
\label{mainTate}
 Let $g$ be a positive integer and let $N=N(g)$ be as in Theorem \ref{mainRelation}.
 
 Let $X$ be a $g$-dimensional abelian variety over a finite field $k$ of characteristic $p$.  Assume that
 there exist a positive integer $n$ and a prime $\ell\ne p$ such that
 the self-product $X^n$ of $X$ carries an exotic $\ell$-adic Tate cohomology class.
 
 Then the self-product $X^{2N}$ of $X$ carries an exotic $l$-adic cohomology Tate class for all primes $l\ne p$.
 \end{thm}

\begin{rem} Theorem \ref{mainTate} gives a positive answer to a question of Kiran Kedlaya, who pointed out
that this result is related to the algorithmic problem of deciding whether or not a given
abelian variety (specified by its Weil polynomial) is {\sl neat} in a sense of \cite[Sect. 3]{ZarhinEssen}, \cite{ZarhinJPAA}).
\end{rem}

Is it possible  to get all  Tate classes on all  self-products of $X$, using only  Tate classses
 of bounded dimension?  In order to answer this question, we need the following result about nonnegative admissible functions.

\begin{thm}
\label{semigroupADM}
Let $g$ be a positive integer. Then there exists a positive even integer $H=H(g)$ that enjoys the following property.

Let $X$ be a $g$-dimensional abelian variety over a finite field $k$.  Then there exist a positive integer $d$
and $d$  nonnegative admissible functions $e_i \colon R_X \to \Z_{+}$ such that:

\begin{itemize}
\item[(i)]
the weight of each $e_i$ does not exceed $H(g)$;
\item[(ii)]
each nonnegative
admissible function $e \colon R_X \to \Z_{+}$ may be presented as a linear combination of $e_1, \dots,  e_d$ with nonnegative integer coefficients.
\end{itemize}
\end{thm}

Theorem \ref{semigroupADM} implies the following assertion.

\begin{thm}
\label{semigroupTate}
Let $g$ be a positive integer. Let $H=H(g)$ be as in Theorem \ref{semigroupADM}.

Let $X$ be a $g$-dimensional abelian variety over a finite field $k$.  Assume that $k$ is sufficiently large w.r.t  $X$.
Let $\ell$ be a prime different from $\fchar(k)$ and
 $n$ be a positive integer. Then every $\ell$-adic Tate cohomology class on $X^n$ may be presented as a linear conbination of products of
 $\ell$-adic Tate cohomology classes of dimension $ \le H(g)$.
%of pullbacks of $\ell$-adic Tate classes on $X^j$ with respect to various projection maps $X^n \to X^j$ with $j \le H(g)$.
\end{thm}

The paper is organized as follows. In Section \ref{neat} we discuss basic useful results  about $R_X$ and related objects,
including the Newton polygons. In addition, we discuss roots of unity in $\Gamma(X,k)$ (Lemma \ref{rootsG})
and the structure and degree of $L_X$ (Lemma \ref{basicL}).
In Section \ref{MultRel} we  study multiplicative relations between Weil numbers (i.e., admissible functions) and their weights; in particular, we prove Lemma \ref{rankGamma} (see Lemma \ref{rank}).  In Sections \ref{RelationProof}   and \ref{GordanL} we prove Theorems  \ref{mainRelation}  and \ref{semigroupADM} respectively.
Section \ref{linAlgebra} contains certain constructions from multilinear algebra that we use
 in Section \ref{TateProof} in order to prove Theorems \ref{mainTate} and \ref{semigroupTate}.

As usual, $\ell$ and $l$ are primes different from $p$, and $\NN, \Z,\Z_{\ell},\Q,\RR, \C,\Q_{\ell}$ stand for the set of positive integers, the rings of integers and $\ell$-adic integers, the fields of rational,  real, complex, and $\ell$-adic  numbers respectively.  We write $\Z_{+}$ and $\RR_{+}$ for  the additive semigroups of {\bf nonnegative} integers and of {\bf nonnegative} real numbers respectively.
If $z$ is a complex number then we write $\bar{z}$ for its complex-conjugate. Similarly, if $\phi \colon E \hookrightarrow \C$ is a field embedding then we write $\bar{\phi}$ for the corresponding complex-conjugate field embedding
$$\bar{\phi} \colon E \hookrightarrow \C, \ x \mapsto \overline{\phi(x)}.$$
If $M$ is a positive integer and $v$ and $w$ are two vectors in $\RR^M$ then we write $v\cdot w$ for their scalar product.
If $A$ is a finite set then we  write $\#(A)$ for the number of its elements. We write $\rk(\Delta)$ for the rank of a finitely generated commutative group $\Delta$. 
%If $\Lambda$ is a ring then we write $\Lambda^{*}$ for the multiplicative group of its invertible elements.

{\bf Acknowledgements}. I am deeply grateful to Kiran Kedlaya for interesting stimulating questions and to the referee, whose comments helped to improve the exposition.

Part of this work was done during my stay at Centre \'Emile Borel (Institut Henri Poincar\'e, Paris) in June-July 2019, whose hospitality and support are gratefully acknowledged.

\section{Preliminaries}
\label{neat}
In this section we discuss basic properties of $L=L_X,  R_X, \Gamma(X,k)$.  Let us start with the formal definition of $\mathcal{P}_X(t)$.

Throughout this paper $k$ is a finite field of characteristic $p$ that consists of $q$ elements, $\bar{k}$ an algebraic closure of $k$ and $\Gal(k)=\Gal(\bar{k}/k)$ the absolute Galois group of $k$. It is well known that the profinite group $\Gal(k)$ is procyclic and the {\sl Frobenius automorphism}
$$\sigma_k \colon \bar{k} \to \bar{k}, \ x \mapsto x^q$$
is a topological generator of $\Gal(k)$. If $\ell \ne p$ is a prime then we write
$$\chi_{\ell} \colon \Gal(k) \to \Z_{\ell}^{*}$$
for the $\ell$-adic {\sl cyclotomic character} that defines the Galois action on all $\ell$-power roots of unity in $\bar{k}$. By definition,
$$\chi_{\ell}(\sigma_k)=q \in  \Z_{\ell}^{*}.$$

Let $X$ be an abelian variety of positive dimension over $k$. We write $\End(X)$ for the ring of its $k$-endomorphisms and $\End^0(X)$ for the corresponding (finite-dimensional semisimple) $\Q$-algebra $\End(X)\otimes\Q$. We write $\Fr_X=\Fr_{X,k}$ for the Frobenius endomorphism of $X$. We have
$$\Fr_X \in \End(X)\subset \End^0(X).$$
It is well known that
\begin{equation}
\label{sigmaFr}
\sigma_k(x)=\Fr_X(x) \ \forall x \in X(\bar{k}).
\end{equation}
By a theorem of Tate \cite[Sect. 3, Th. 2 on p, 140]{Tate1}, the $\Q$-subalgebra $\Q[\Fr_X]$ of $\End^0(X)$ generated by $\Fr_X$ coincides with the center of $\End^0(X)$. In particular, if $\End^0(X)$ is a field then $\End^0(X)=\Q[\Fr_X]$.

 If $\ell$ is a prime different from $p$ then we  write $T_{\ell}(X)$ for the $\Z_{\ell}$-Tate module of $X$ and $V_{\ell}(X)$ for the corresponding $\Q_{\ell}$-vector space
$$V_{\ell}(X)=T_{\ell}(X)\otimes_{\Z_{\ell}}\Q_{\ell}.$$
It is well known \cite[Sect. 18]{Mumford} that $T_{\ell}(X)$ is a free $\Z_{\ell}$-module of rank $2\dim(X)$ that may be viewed as a $\Z_{\ell}$-lattice in the $\Q_{\ell}$-vector space $V_{\ell}(X)$   of dimension $2\dim(X)$. The Galois action on $X(\bar{k})$ induces the continuous group homomorphism \cite{SerreAbelian,SerreRibet}
$$\rho_{\ell}=\rho_{\ell,X} \colon \Gal(K) \to \Aut_{\Z_{\ell}}(T_{\ell}(X))\subset \Aut_{\Q_{\ell}}(V_{\ell}(X)).$$
In addition, there is a canonical isomorphism of $\Gal(k)$-modules
$X[\ell] \cong T_{\ell}(X)/\ell$ where $X[\ell]$ is the kernel of multiplication by $\ell$ in $X(\bar{k})$.

 By functoriality, $\End(X)$ and $\Fr_X$ acts
  on ($T_{\ell}(X)$ and) $V_{\ell}(X)$; it is well known that the action of $\Fr_X$ coincides with the action of $\rho_{\ell}(\sigma_k)$. By a theorem of A. Weil \cite[Sect. 19 and Sect. 21]{Mumford}, $\Fr_X$ acts on $V_{\ell}(X)$ as a semisimple linear operator, its characteristic polynomial
$$\P_X(t)=\P_{X,k}(t)=\det (t \I -\Fr_X, V_{\ell}(X)) \in \Z_{\ell}[t]$$
lies in $\Z[t]$ and does not depend on a choice of $\ell$. In addition, all eigenvalues of $\Fr_X$ (which are algebraic integers) have archimedean absolute value equal to $q^{1/2}$,
and  if an eigenvalue of $\Fr_X$  is a square root of $q$ then its multiplicity is even (see \cite[p. 267]{ZarhinK3}).
This implies that the constant term of $\P_{X,k}(t)$ is $q^{\dim(X)}$.  In particular, $\Fr_X$ acts as an {\sl automorphism} of the free $\Z_{\ell}$-module $T_{\ell}(X)$.

  This means that if
$$L=L_X \subset \C$$
is the splitting field of $\P_X(t)$ and  $$R_X=R_{X,k} \subset L$$ is the set of roots of $P(t)$ then $L$ is a finite Galois extension of $\Q$ such that for every field embedding $L \hookrightarrow \C$ we have $\mid \alpha \mid =q^{1/2}$ for all $\alpha \in R_X$.  Let  $\Gal(L/\Q)$ be the Galois group of $L/\Q$.
 Clearly, $R_X$ is a $\Gal(L/\Q)$-invariant (finite) subset of $L^{*}$. It follows easily that if  $\alpha \in R_X$ then $q/\alpha \in R_X$. Indeed,
$q/\alpha$ is the {\sl complex-conjugate} $\bar{\alpha}$ of $\alpha$. We have
$$q^{-1}\alpha^2=\frac{\alpha}{q/\alpha}.$$

\begin{defn}
Let $\ell$ be a prime and $n$ a positive integer. We write $\mathbf{e}_{\ell}(n)$
for the largest order of elements of the general linear group $\GL(n,\F_{\ell})$.
We write $\exp_{\ell}(n)$ for the {\sl exponent} of  $\GL(n,\F_{\ell})$.
\end{defn}

Recall that $\Gamma(X,k)$ is the multiplicative subgroup of $L$ generated by $R_X$.

\begin{lem}
\label{rootsG}
If $\gamma \in \Gamma(X,k)$ is a root of unity then there is a positive integer
$m \le \max(2 \mathbf{e}_2(2g), \mathbf{e}_3(2g))$ such that $\gamma^m=1$. In addition,
$\gamma^{D(g)}=1$ where 
$$D(g):=\mathrm{LCM}(2\exp_{2}(2g),\exp_{3}(2g)), \ \text{ and }  \ m | D(g).$$
\end{lem}

\begin{proof}
 In what follows we choose a prime $\ell \ne p$ and view $\Fr_X$ as the automorphism of   free $\Z_{\ell}$-module $T_{\ell}(X)$ of rank $2g$. Then
 $\Fr_X$ induces the automorphism $\Fr_X \bmod \ell$ of the $2g$-dimensional $\F_{\ell}$-vector space 
 $$T_{\ell}(X)/\ell=X[\ell].$$
 Let $r$ be the order of 
 $$\Fr_X \bmod \ell \in \Aut_{\F_{\ell}}(X[\ell]) \cong \GL(2g, \F_{\ell}).$$ Clearly,
 $$r \le \mathbf{e}_{\ell}(2g) \  \text{ and } r \ | \exp_{\ell}(2g).$$
In addition,
$$\Fr_X^r \in \mathrm{Id}+\ell \End_{\Z_{\ell}}(T_{\ell}(X)).$$
Let $\Delta$ be the multiplicative group generated by all the eigenvalues of $\Fr_X^r$.
Clearly, $\delta=\gamma^r\in \Delta$.
Applying a variant of Minkowski's Lemma   \cite[Lemma 2.4]{SZcomp},
we obtain that $\delta=1$ if $\ell>2$ and $\delta^2=1$ if $\ell=2$.
This implies that $\gamma^r=1$ if $\ell>2$ and $\gamma^{2r}=1$ if $\ell=2$.
Now let us put $\ell=2$ if $p\ne 2$ and $\ell=3$ if $p=2$.
The rest is clear.

\end{proof}

Let $\Oc_L$ be the ring of integers in $L$.   Clearly, $R_X \subset \Oc_L$.
By a classical  theorem of A. Weil (Riemann's hypothesis) \cite{Mumford}, if $j \colon L_X=L \hookrightarrow \C$ is a field  embedding then $j(\alpha)\overline{j(\alpha)}=q$. 
 This implies that if $\B$ is a maximal ideal in $\Oc_L$ such that $\fchar(\Oc_L/\B) \ne p$  then all elements of $R_X$ are $\B$-adic units.
  The $p$-adic behaviour of $R_X$ is described in terms of the set $\Sl_X$ of {\sl slopes of the Newton polygon} of $X$ \cite{OortG} (see also \cite[Sect. 4]{ZarhinJPAA}). Recall that
   $\Sl_X$ is a finite nonempty set of rational numbers that enjoys the following properties.
  
  \begin{itemize}
  \item[(i)]  $0 \le c \le 1$ for all $c\in \Sl_X$.
   \item[(ii)] $c\in \Sl_X$ if and only if $1-c\in \Sl_X$.
    \item[(iii)]  If $c\in \Sl_X$ then either $c=1/2$ or there is a positive integer $h \le g=\dim(X)$ such that $c\in \frac{1}{h}\Z$. 
    \item[(iv)]  Let $\mathfrak{P}$ 
be any maximal ideal in $\Oc_L$ such that $\fchar(\Oc_L/\mathfrak{P})= p$ and let
$$\ord_{\PP} \colon L^{*} \to \Q$$
be the discrete valuation map attached to $\PP$ that is normalized by the condition
$$\ord_{\PP}(q)=1.$$
Then
$$\ord_{\PP}(R_X)=\Sl_X.$$
\item[(v)]  If $\alpha\in R_X$ then 
$$\ord_{\PP}(q/\alpha)=1-\ord_{\PP}(\alpha).$$
\item[(vi)] Let $\mu_{L}$ the multiplicative group of all roots of
unity in $L$. Then its image $\ord_{\PP}(\mu_{L})=\{0\}$.
  \end{itemize}
  
  Properties (i)-(vi) imply readily the following assertion.
  
  \begin{lem}
  \label{SLP}
  Let $g$ be a positive integer.
  Let us consider the  set $\mathrm{Slp}(g)$ of all rational numbers $c$  that enjoy the following properties.
  \begin{enumerate}
  \item
  $0 \le c \le 1$.
  \item
  Either $c=1/2$ or there  exists a positive integer $h \le g$ such that $c\in \frac{1}{h}\Z$. 
  \end{enumerate}
  Then $\mathrm{Slp}(g)$ is a {\sl finite nonempty} set that  enjoys the following property.
  
  {\sl If $X$ is a $g$-dimensional abelian variety over a finite field then $\Sl_X$ lies in} $\mathrm{Slp}(g)$.
  \end{lem}
  
  \begin{rems}
  \label{ordP}
  Let $S(p)$ be the set of all maximal ideals in $\Oc_L$ such that $\fchar(\Oc_L/\mathfrak{P})= p$. Since $L/\Q$ is Galois,  $\#(S(p))$ divides $[L:\Q]$;
  in particular, $\#(S(p))$ divides $g! 2^g$ (see Lemma \ref{basicL}(ii) below). Let us define a group homomorphism
  $$w_X \colon \Gamma(X,K) \to \Q^{S(p)},   \gamma \mapsto \{\ord_{\PP}(\gamma)\}_{\PP\in S(p)}.$$
  \begin{itemize}
  \item[(a)]
  Clearly,
  $w_X(q)$ is the vector ${\bf 1}\in \Q^{S(p)}$, all whose coordinates equal $1$, hence
   $$w_X\left(\frac{q}{\alpha}\right)={\bf 1}-w_X(\alpha) \ \forall \alpha \in R_X.$$
   \item[(b)]
  By Property (iv) and Lemma \ref{SLP}, 
  $$w_X(R_X) \subset \Sl_X^{S(p)} \subset \mathrm{Slp}(g)^{S(p)} \subset \Q^{S(p)}.$$
  \item[(c)]
  It follows from Property (v) that a vector $\tilde{c} \in  \Q^{S(p)}$ lies in $w_X(R_X)$ if and only if
  ${\bf 1}-\tilde{c} \in w_X(R_X)$.
   \item[(d)]
   In light of Property (vi), $w_X(\gamma)=0$ if $\gamma$ is a root of unity. The converse is also true:
  it is proven in \cite[Prop. 2.1 on p. 249]{ZarhinTokyo} (see also \cite[Prop. 3.1.5]{ZarhinInv79}) that $\ker(w_X)$ {\sl consists of roots of unity}.
   \item[(e)] 
  It follows readily from (d)  that:
  
  \begin{enumerate}
  \item[(1)] none of elements in $R_X$ lies in $\ker(w_X)$;
  \item[(2)] 
  if $k$ is {\sl not small} w.r.t $X$ and $\alpha_1, \alpha_2$ are distinct elements of $R_X$ then
  $$w_X(\alpha_1) \ne w_X(\alpha_2).$$
  \end{enumerate}
  \end{itemize}
  \end{rems}
  
  \begin{lem}
\label{basicL}
\begin{itemize}
\item[]
\item[(i)]
The field  $L_X$ is either $\Q$ or $\Q(\sqrt{p})$ or  a CM field.
\item[(ii)]
The field $L_X$ is a finite Galois extension of $\Q$ and its degree $[L_X:\Q]$ divides $g! 2^g$.
\item[(iii)] $\#(S(p))$ divides $g! 2^g$.
\end{itemize}
\end{lem}

\begin{proof}
Let us prove (i,ii).
By definition of the splitting field,  $L_X/\Q$ is Galois.

Suppose that $X$ is simple.  According to \cite{Tate1,Tate2},
$\mathcal{P}_X(t)$ is a power $\mathcal{P}_{\mathrm{irr}}(t)^{a}$ of a $\Q$-irreducible monic
polynomial $\mathcal{P}_{\mathrm{irr}}(t)$ where $a$ is a positive integer dividing $2g$ and
$$\deg(\mathcal{P}_{\mathrm{irr}})=\frac{2g}{a}.$$
Clearly, $L_X$ is the normal closure of the degree $2g/a$ number field $E_X:=\Q[t]/\mathcal{P}_{\mathrm{irr}}(t)$.
According to \cite[Exemples]{Tate2}, $E_X$ is either  $\Q$ or $\Q(\sqrt{p})$ or  a CM field.

In the first two cases $L_X=\Q$ or $\Q(\sqrt{p})$; in particular, it is a totally real number field, whose
degree divides $2g=2\dim(X)$. 

In the third case let $E_X^{+}$ be the maximal totally real  subfield of $E_X$;
$[E_X^{+}:\Q]=g/a$ and $E_X$ is a purely imaginary quadratic extension $E_X^{+}(\sqrt{-\delta})$ of $E_X^{+}$.
Here $\delta$ is a totally positive element of $E_X^{+}$. 
Let $L_X^{+}$ be the normal closure of $E_X^{+}$. Since $E_X^{+}$ is totally real,  $L_X^{+}$ is  totally real as well, and its degree $[L_X^{+}:\Q]$ divides 
$[E_X^{+}:\Q]!=(g/a)!$. Since $-\delta \in E_X^{+}\subset L_X^{+}$, its Galois orbit in $L_X^{+}$ consists at most
of  $[E_X^{+}:\Q]=g/a$ elements. This implies that  $[L_X: L_X^{+}]$ divides $2^{g/a}$, since $L_X$ is obtained from 
 $L_X^{+}$ by adjoining square roots of all the (totally negative) Galois conjugates $-\delta$. This implies that
 $L_X$ is a CM field and $[L_X:\Q]$ divides $(g/a)!\cdot 2^{g/a}$, which in turn, divides $g! \cdot 2^g$. 
 
 Now let us consider the general case when $X$ is isogenous to a product $\prod_{i=1}^m X_i$
 of $m$ nonzero  simple abelian varieties $X_i$. 
 It is well known that if we put 
 $$g_i:=\dim(X_i),   \ \text{ and }  L_i:=L_{X_i}$$
  then
 $$g=\dim(X)=\sum_{i=1}^m g_i, \  \mathcal{P}_{X}(t)=\prod_{i=1}^m \mathcal{P}_{X_i}(t).$$
 Let $\bar{\Q}$ be an algebraic closure of $\Q$. We may and will view all $L_i$ as subfields of $\bar{\Q}$. Then
 $L_X$ is the {\bf compositum} of $m$ number fields $L_{X_i}=L_i$ in $\bar{\Q}$. Applying (i,ii) to simple $X_i$'s,
 we obtain that $L_X$ is either $\Q$ or $\Q(\sqrt{p})$ or a CM field, which proves (i).
 
 In order to prove (ii), recall that all $L_{i}/\Q$  and $L_X/\Q$ are finite Galois extensions.
 Let $\Gal(L_{i}/\Q)$ and $\Gal(L_{X}/\Q)$ be the corresponding (finite) Galois groups.
 Clearly, each $L_{i}$ is a $\Gal(L_{X}/\Q)$-invariant subfield of $L_X$, and the corresponding restriction map
 $$\Gal(L_{X}/\Q) \to \Gal(L_{i}/\Q), \ s \mapsto s_i$$
 is a surjective group homomorphism. On the other hand, since all the $L_i$'s generate $L_X$ as a field, the product-map
 $$\Gal(L_X/\Q) \to \prod_{i=1}^m \Gal(L_{i}/\Q), \  s \mapsto \{s_i\}_{i=1}^m$$
 is a {\sl group embedding}. By Lagrange's theorem, $\#(\Gal(L_X/\Q) )$ divides  $\prod_{i=1}^m \#(\Gal(L_{i}/\Q))$. In other words,
  $[L_X:\Q]$ divides 
 $\prod_{i=1}^m [L_i:\Q]$, which, in turn, divides 
 $$\prod_{i=1}^m g_i! 2^{g_i}=2^g \prod_{i=1}^m g_i!.$$
 Since $\sum_{i=1}^m g_i=g$, the product $\prod_{i=1}^m g_i!$ divides $g!$. This implies that
 $[L_X:\Q]$ divides $2^g g!$, which ends the proof of (ii).
 
Let us prove (iii). Since $L_X/\Q$ is Galois,  $\#(S(p))$ divides $[L_X:\Q]$. Now (iii) follows readily from (ii).
\end{proof}

  \section{Multiplicative relations between Weil numbers}
  \label{MultRel}
  
  This section contains auxiliary results that will be used in Section \ref{RelationProof}  in the proof of Theorem \ref{mainRelation}. 
  
  \begin{sect}
  \label{involution}
  {\bf The involution}.
  Recall that there is an involution map
  \begin{equation}
  \label{iotaR}
  \iota: R_X \to R_X, \ \alpha \mapsto \frac{q}{\alpha}=\bar{\alpha}.
  \end{equation}
  Let $R_X^{\iota}$ be the subset of fixed points of $\iota$. Its elements
   (if there are any) are square roots of $q$; hence, 
   \begin{equation}
   \label{iota2}
   \#(R_X^{\iota})\le 2.
   \end{equation}
   In addition, if $k$ is {\sl not} small with respect to $X$ then at most {\sl one square root} of $q$ lies in $R_X$, hence,
   \begin{equation}
   \label{iota1}
   \#(R_X^{\iota})\le 1.
   \end{equation}
   
   \begin{rem}
   \label{ssF}
   Suppose that $k$ is {\sl not} small w.r.t $X$.
   If $\beta$ is an element of $R_X$ such that $\beta^2/q$ is a root of unity then
   $q/\beta \in R_X$ and the ratio
   $$\frac{\beta}{q/\beta}=\frac{q}{\beta^2}$$
   is a root of unity. This implies that $\beta=q/\beta$, i.e., $\beta \in R_X^{\iota}$.
   \end{rem}
   
   Let us consider the free abelian group $\Z^{R_X}$ of functions $e: R_X \to \Z$. The involution $\iota$
   induces an automorphism (also an involution)
   $$\iota^{*} \colon \Z^{R_X} \to \Z^{R_X}, \ \iota^{*}e(\alpha):=e(\iota \alpha)=e(q/\alpha).$$
   Let us consider the group homomorphism
   $$\Pi \colon \Z^{R_X} \to \Gamma(X,k)\subset L_X^{*},  \ e \mapsto \prod_{\alpha\in R_X}\alpha^{e(\alpha)}.$$
   We have
   $$\Pi(\iota^{*}e)=\overline{\Pi(e)}.$$

   \begin{rems}
   \label{supersing}
   \begin{itemize}
   \item[]
   \item[(i)]
   If $\iota^{*}e=e$ then $\Pi(e)=\overline{\Pi(e)}$ is  (totally) real; it follows from Weil's theorem that
   $\Pi(e)^2$ is an integral power of $q$. In particular, if $\sum_{\alpha\in R_X}e(\alpha)$ is even then it follows
   from Weil's theorem that $\Pi(e)$ is $\pm$ integral power of $q$.
   \item[(ii)]
   If $f$ is a function $R_X \to \Z$ then the function $e:=f+\iota^{*}f$ is obviously trivial.
   
   Conversely, one may easily check that a function $e \colon R_X \to \Z$ is trivial if and only if there  exists $f \colon R_X \to \Z$ such that
   $e=f+\iota^{*}f$.
   \item[(iii)]
   Clearly, $e$ is admissible if and only if $\Pi(e)$ lies in the cyclic multiplicative subgroup $q^{\Z}$ generated by $q$.
   \end{itemize}
   \end{rems}
   \end{sect}
   
   \begin{sect}
   \label{Ror}
   {\bf Ranks and Orbits}.
   The complement $R_X\setminus R_X^{\iota}$ splits (if it is not empty)  into a disjoint union of 
   $2$-element orbits of $\iota$ say $\{\alpha,q/\alpha\}$. Let $r_X$ be the number of such orbits, which is a nonnegative integer that vanishes
   if and only if $R_X = R_X^{\iota}$. We have
   \begin{equation}
   \label{rXdef}
    \#(R_X\setminus R_X^{\iota})=2 r_X; \  r_X \le \frac{\#(R_X)}{2} \le \frac{2g}{2}=g.
    \end{equation} 
    If $r_X\ge 1$ (i.e., $R_X \ne R_X^{\iota}$)  then we  have $r_X$ $2$-elements $\iota$-orbits $O_1, \dots, O_{r_X}$ in $R_X\setminus R_X^{\iota}$.
    By choosing arbitrarily an element $\alpha_i \in O_i$ for all $i=1, \dots, r_X$, we get
    \begin{equation}
    \label{OR}
    O_i=\{\alpha_i,q/\alpha_i\} \ \forall i=1, \dots, r_X; \  R_X \setminus R_X^{\iota}=\{\alpha_1, q/\alpha_1, \dots, \alpha_{r_X},q/\alpha_{r_X}\}.
    \end{equation}
   
   Recall \cite{ZarhinEssen}, that $\Gamma(X,k)$ always contains $q$.  Since $\beta^2=q$ for all $\beta \in R_X^{\iota}$,
   the subgroup $\Gamma_1(X,k)$ of $\Gamma(X,k)$ generated by $q$ and all elements  of $R_X\setminus R_X^{\iota}$  has finite index
   in  $\Gamma(X,k)$. In particular,
   \begin{equation}
   \label{rkGamma1}
   \rk(\Gamma_1(X,k))= \rk(\Gamma(X,k)).
   \end{equation}
   
   Clearly, if $r_X=0$ then $\Gamma_1(X,k)=q^\Z$ has rank $1$, hence $\rk(\Gamma(X,k))=1$.
   It follows from \eqref{OR} that if $r_X \ge 1$ then $\Gamma_1(X,k)$ is generated by $\{\alpha_1, \dots, \alpha_{r_X};q\}.$
   In particular,
   \begin{equation}
   \label{rankGamma1rX}
    \rk(\Gamma_1(X,k)) \le r_X+1.
   \end{equation}
   
   \begin{rem}
   \label{rankSmall}
   Suppose that $k$ is {\sl not small} w.r.t $X$. Then $\#(R_X^{\iota})=0$ or $1$ and therefore
   $\#(R_X)=2 r_X$ or $2 r_X+1$ respectively. In both cases
   \begin{equation}
   \label{rXR}
   r_X=\lfloor\frac{\#(R_X)}{2}\rfloor.
   \end{equation}
   Combining \eqref{rXR} with \eqref{rankGamma1rX} and \eqref{rkGamma1}, we obtain that
   \begin{equation}
   \label{smallrXr}
   \rk(\Gamma(X,k))= \rk(\Gamma_1(X,k))\le r_X+1=\lfloor\frac{\#(R_X)}{2}\rfloor+1.
   \end{equation}

   \end{rem}
   \end{sect}
   
   \begin{sect}{\bf Nontrivial and reduced admissible functions}.
   \label{nontrivial}
   The existence of a nontrivial admissible function implies certain restrictions on $R_X$.
 
  \begin{lem}
 \label{multiply}
 Suppose that there exists a nontrivial admissible function $e \colon R_X \to \Z$ of degree, say, $d$.
 
 Then the following conditions hold.
 
 \begin{itemize}
 \item[(i)] $R_X \ne R_X^{\iota}$, i.e.,
 $R_X \setminus R_X^{\iota}$ is a nonempty subset of $R_X$.

  \item[(ii)]
  For each nonzero integer $m$ the function
 $$m \cdot e \colon R_X \to \Z, \  \alpha \mapsto m\cdot e(\alpha)$$
 is also nontrivial admissible.
   \item[(iii)]
   Let us consider the function
   $e_0 \colon R_X \to \Z$ that vanishes identically on $R_X^{\iota}$ (if this subset is nonempty)
   and coincides with $e$ on $R_X \setminus R_X^{\iota}$. Then $e_0$ is nontrivial.
   In addition,
    for each nonzero even integer $m$
   the function 
   $$m \cdot e_0: R_X \to \Z, \ \alpha \mapsto m\cdot e_0(\alpha)$$
   is nontrivial admissible, and its weight

   \begin{equation}
   \label{wte0}
   \wt(m\cdot e_0)=|m|\wt(e_0) \le |m|\wt(e).
   \end{equation}

 \end{itemize}
 \end{lem}
 
 \begin{proof}
 If  $R_X^{\iota}=\emptyset$ then all three assertions of Lemma are obviously true.  So, let us assume that
 $R_X^{\iota}\ne\emptyset$.
 
 (i) Suppose that $R_X=R_X^{\iota}$. Then 
 $$\prod_{\alpha\in R_X^{\iota}}\alpha^{e(\alpha)} =q^d  \text{ and } \sum _{\alpha\in R_X^{\iota}}e(\alpha)=2d$$
 is an even integer. Since $ R_X^{\iota}$ consists of one or two elements and $e$ is {\sl nontrivial}, there is $\beta \in  R_X^{\iota}$
 such the $e(\beta)$ is odd. This implies that  $ R_X^{\iota}$ consists of two elements, say, $\beta$ and $-\beta$,
 $$e(\beta)+e(-\beta)=2d$$
   and both integers
 $e(\beta)$ and $e(-\beta)$ are {\sl odd}. This implies that (recall that $\beta^2=q$)
 $$q^d =\beta^{e(\beta)} \cdot (-\beta)^{e(-\beta)}=\beta^{e(\beta)} \cdot (-1)\cdot \beta^{e(-\beta)}=
 -\beta^{e(\beta)+e(-\beta)}=-\beta^{2d}=-q^d.$$
 So, $q^d=-q^d$, which is absurd. The obtained contradiction proves (i).
 
 (ii).
 The admissibility of $m\cdot e$ is obvious.  The nontriviality is also clear  if there exists $\alpha \in R_X\setminus R_X^{\iota}$ with 
 $$e(\alpha) \ne e(q/\alpha).$$
 So, we may assume that  $m\cdot e$ is trivial (we are going to arrive to a contradiction),
  and
 \begin{equation}
 \label{push}
 e(\alpha)=e(q/\alpha) \ \forall \alpha\in R_X\setminus R_X^{\iota}. 
 \end{equation}
 This implies that
 there is an integer $n$ such that
 $$\prod_{\alpha\in R_X^{\iota}}\alpha^{e(\alpha)}=q^n.$$
 It follows  that
 the sum 
 \begin{equation}
 \label{TwoN}
 \sum_{\alpha\in R_X^{\iota}}e(\alpha)=2n\in 2\Z
 \end{equation}
  is an {\sl even} integer. On the other hand,
 the nontriviality of $e$ combined with \eqref{push} implies that there is $\beta \in R_X^{\iota}$ with {\sl odd} $e(\beta)$.
 Since $\#(R_X^{\iota})\le 2$, it follows from \eqref{TwoN} that  integer $e(\alpha)$ is odd for all $\alpha \in  R_X^{\iota}$.
 It follows that
 $\prod_{\alpha\in R_X^{\iota}}\alpha =q^d$ for some integer $d$. Therefore $\#(R_X^{\iota})=2d$ is a positive even integer,
 i.e., $R_X^{\iota}$ consists of two elements  $\beta,-\beta$ with $\beta^2=q$; in addition, both integers 
 $e(\beta)$ and $e(-\beta)$ are odd. The same computations as in the proof of (i)  give us that
 $$q^d=\beta^{e(\beta)} (-\beta)^{e(-\beta)}=-q^d,$$
 hence, $q^d=-q^d$. The obtained contradiction proves  the nontriviality of $m \cdot e$.
 
 (iii)  Suppose that $e_0$ is trivial, i.e.,
 $$e(\alpha)=e(q/\alpha) \ \forall \alpha \in R_X\setminus R_X^{\iota}.$$
  Then it is admissible and therefore there is an integer $h$ such that 
 $$\prod_{\alpha \in R_X\setminus R_X^{\iota}}\alpha^{e_0(\alpha)}=\prod_{\alpha \in R_X\setminus R_X^{\iota}}\alpha^{e(\alpha)}=q^h.$$
  Since $e$ is admissible of degree $d$,
 $$ \prod_{\beta \in R_X^{\iota}}\beta^{e(\beta)}=q^{d-h}.$$
  The nontriviality of $e$ implies that there is $\beta \in R_X^{\iota}$ such that  integer $e(\beta)$ is {\sl odd}.
 Now the same computations as in the proof of (i) give us that $R_X^{\iota}$ consists of two elements $\beta$ and $-\beta$,
 both integers $e(\beta)$ and $e(-\beta)$ are odd and eventually, $q^{d-h}=-q^{d-h}$. The obtained contradiction proves that $e$ is {\sl nontrivial},
 which is the first assertion of (iii).
 
 Let us prove the second asssertion of (iii). Since $m$ is even, there is an integer $n$ such that $m=2n$.We have
 $$q^{md}=\left(\prod_{\alpha \in R_X}\alpha^{e(\alpha)}\right)^m=\left(\prod_{\alpha \in R_X\setminus R_X^{\iota}}\alpha^{e(\alpha)}\right)^m
 \times \left(\prod_{\beta \in  R_X^{\iota}}\beta^{ e(\beta)}\right)^m=$$
 $$\left(\prod_{\alpha \in R_X}\alpha^{m\cdot e_0(\alpha)}\right) 
 \times \left(\prod_{\beta \in  R_X^{\iota}}\beta^{2n\cdot e(\beta)}\right)=$$
 $$\left(\prod_{\alpha \in R_X}\alpha^{m \cdot e_0(\alpha)}\right) 
 \times \left(\prod_{\beta \in  R_X^{\iota}} q^{n\cdot e(\beta)}\right).$$
 This implies that $\prod_{\alpha \in R_X}\alpha^{m \cdot e_0(\alpha)}$ is an integral power of $q$,
 i.e., $m \cdot e_0$ is {\sl admissible}. The nontriviality of $m \cdot e_0$ follows from the nontriviality of $e_0$,
 because $e_0$ vanishes identically on $R_X^{\iota}$. This ends the proof of (iii).
 The last assertion of (iii) about weights follows readily from obvious inequality $\wt(e_0) \le \wt(e).$
 \end{proof}
  
  It turns out that one may easily construct a reduced admissible function when $k$ is {\sl small} w.r.t $X$.
  
  \begin{lem}
  \label{rootsR}
  Assume that there are distinct $\alpha_1, \alpha_2 \in R_X$ such that $\gamma:=\alpha_2/\alpha_1$
  is a root of unity.  Then there is a reduced admissible function $e \colon R_X \to \Z$,
  whose weight $w$  enjoys the following properties.
  $$w \le 4 \mathbf{e}_2(2g) \ \text{ if } p \ne 2; \  w \le 2 \mathbf{e}_3(2g)) \text{ if } p = 2.$$
  \end{lem}

  \begin{proof}
  %{\bf Step 0}.
  %Assume that there are distinct $\alpha_1, \alpha_2 \in R_X$ such that $\gamma=\alpha_2/\alpha_1$
 % is a root of unity. 
 Clearly, $\gamma \in \Gamma(X,k)$. 
   By Lemma \ref{rootsG}, there is a positive integer $m$ such that 
 $$\gamma^m=1; \ m  \le 2 \mathbf{e}_2(2g) \ \text{ if } p \ne 2; \  m \le  \mathbf{e}_3(2g)) \text{ if } p = 2.$$
 Hence, it suffices to produce a reduced multiplicative relation of weight $2m$.
 To this end, notice that $q/\alpha_1\in R_X$ and
 $$\alpha_2^m (q/\alpha_1)^m=q^m.$$
 If $\alpha_2 \ne q/\alpha_1$ then we may define
 $$e \colon R_X \to \Z, \ e(\alpha_2):=m, e(q/\alpha_1):=m;   \ e(\alpha):=0 \ \text{ for all other } \alpha.$$
 Clearly, $e$ is a {\sl reduced admissible} function  of weight $2m$.
 
 Suppose that  $\alpha_2 = q/\alpha_1$. Since $\alpha_1 \ne \alpha_2,$
 $$\alpha_1 \ne q/\alpha_1, \ \alpha_1^2 \ne q.$$
 Then we have
 $$q^m=(\alpha_1 \alpha_2)^m=\alpha_1^{2m}, \ 
 \text{ i.e., }
 \alpha_1^{2m}=q^m.$$
 Now let us consider 
 $$e \colon R_X \to \Z, \ e(\alpha_1):=2m;   \ e(\alpha):=0 \ \text{ for all other } \alpha.$$
 Clearly, $e$ is a {\sl reduced admissible} function of weight 
 $2m$.
  \end{proof}
 
  The next Lemma asserts that the existence of a nontrivial admissible function implies the existence of a reduced admissible function,
  whose weight we can control.

 \begin{lem}
  \label{nonTrivRed}
  Let $w$ be a positive integer.
  Suppose that $k$ is not small w.r.t $X$ and there exists a nontrivial admissible function of weight $\le w$.
  
 Then there exist a nonempty subset $A_1\subset R_X$, an integer-valued function
 $\tilde{e} \colon A_1 \to \Z$, and a positive integer $s\le w$ that enjoy the following properties.
 \begin{enumerate}
 \item[(1)] 
 $\forall \alpha \in A_1$ we have $\frac{q}{\alpha}\not\in A_1, \ \tilde{e}(\alpha)>0$.
 \item[(2)] 
 \begin{equation}
 \label{A1R}
 \prod_{\alpha\in A_1}\alpha^{\tilde{e}(\alpha)}=q^s.
 \end{equation}
 In particular,  if we define
 $$f \colon R_X \to \Z, \ f(\alpha):=\tilde{e}(\alpha) \ \forall \alpha\in A_1; \ f(\alpha):=0 \ \forall \alpha\not\in A_1$$
 then $f$ is a reduced admissible function of weight $2s\le 2w$ that vanishes identically on $R_X^{\iota}$.
 %\eqref{A1R} is a reduced multiplicative relation between elements of $R_X$ of weight $2s$.
 \end{enumerate}
  
  %\end{itemize}
  \end{lem}
  
  \begin{proof}
% In light of Lemma \ref{rootsR}, 
%we may and will assume that $k$ is {\sl not} small wrt $X$. 
%In particular, there is at most one $\beta \in R_X$ with $\beta^2=q$. B Then 
By  Lemma \ref{multiply},   $R_X\setminus R_X^{\iota}$ is {\sl not} empty.
Let $e \colon R_X \to \Z$ be a {\sl nontrivial} admissible function $e$ of weight $\le w$.
Let us consider (in the notation of Lemma \ref{multiply}) the function 
$$h_2=2\cdot e_0 \colon R_X \to \Z.$$
It follows from Lemma \ref{multiply} that $h_2$ is nontrivial admissible, 
it vanishes identically on $R_X^{\iota}$  and its weight does not exceed $2w$. This implies that 
 \begin{equation}
 \label{needReduced}
 \prod_{\alpha\in R_X\setminus R_X^{\iota}} \alpha^{h_2(\alpha)}=q^d, \ 2w \ge \sum_{\alpha\in R_X}|h_2(\alpha)|
 \end{equation}
 where $d$ is an integer such that 
 $$|d| \le \wt(h_2) \le 2w.$$
 
The nontriviality and vanishing everywhere at $R_X^{\iota}$ of $h_2$ imply that the subset $A$ of $R_X$ defined by
$$A: =\{\alpha \in R_X\mid h_2(\alpha) \ne h_2(q/\alpha)\}$$
is {\sl nonempty}.  It follows from the very definition that $A$ is $\iota$-invariant and does {\sl not} meet $R_X^{\iota}$.
Let us define the subset $A_1$ of $A$ by 
 $$A_1:=\{\alpha\in A\mid h_2(\alpha)>h_2(q/\alpha)\} \subset A\subset R_X.$$
 Clearly, if $\alpha \in A$ then $\alpha\in A_1$ if and only if $\bar{\alpha}=q/\alpha\not\in A_1$. This implies that
 $A_1$ is {\sl nonempty} and $A$ is the {\sl disjoint union} of $A_1$ and $\iota(A_1)$. In particular, 
 $\#(A)=2\#(A_1)$.
 
 On the other hand, if 
 $$B:=\{\beta\in R_X\setminus R_X^{\iota}\mid h_2(\beta)=h_2(q/\beta)\}\subset  R_X\setminus R_X^{\iota}\}$$
 then $B$ is $\iota$-invariant, $R_X\setminus R_X^{\iota}$ is a disjoint union of $A$ and $B$, and
 $$\prod_{\beta\in B}\beta^{h_2(\beta)}=q^n$$
 for some integer $n$ with
 $$|n| \le \wt(h_2) \le 2w.$$
  Since $R_X\setminus R_X^{\iota}$ is a disjoint union of $A$ and $B$, it follows from
 \eqref{needReduced} that
 $$\prod_{\alpha\in A}\alpha^{h_2(\alpha)}=\frac{q^d}{q^n}=q^{d-n}.$$
 Since $A$ is a disjoint union of $A_1$ and $\iota(A_1)$, we get
 $$q^{d-n}=\left(\prod_{\alpha\in A_1}\alpha^{h_2(\alpha)}\right)\times \left(\prod_{\alpha\in A_1}\iota(\alpha)^{h_2(\iota \alpha)}\right)=$$
 $$\left(\prod_{\alpha\in A_1}\alpha^{h_2(\alpha)}\right)\times \left(\prod_{\alpha\in A_1}(q/\alpha)^{h_2(q/ \alpha)}\right)=$$
 $$\left(\prod_{\alpha\in A_1}\alpha^{h_2(\alpha)-h_2(q/\alpha))}\right)\times q^m$$
 where $m:=\sum_{\alpha\in A_1}h_2(q/\alpha)\in \Z$. 
 If we define the function
 $$\tilde{e} \colon A_1 \to \Z,  \ \alpha\mapsto h_2(\alpha)-h_2(q/\alpha)$$
  then $\tilde{e}(\alpha)>0 \ \forall \alpha \in A_1$,
  $$\sum_{\alpha \in A_1}\tilde{e}(\alpha)\le \sum_{\alpha\in A_1}\left(|h_2(\alpha)|+|h_2(q/\alpha)|\right)=\sum_{\alpha\in A}|h_2(\alpha)|\le \wt(h_2)\le 2w,$$
   and
 $$q^{d-n}=\left(\prod_{\alpha\in A_1}\alpha^{\tilde{e}(\alpha)}\right)\times q^m,$$
 i.e.,
 $$\prod_{\alpha\in A_1}\alpha^{\tilde{e}(\alpha)}=q^{d-n-m}.$$
 It remains to put $s:=d-n-m$.
 This ends the proof.
 \end{proof}
 \end{sect}
 
 The following assertion contains Lemma \ref{rankGamma}. (Recall  that   $\Gamma_1(X,k)$ is defined in Subsection \ref{Ror}.)
 
 \begin{lem}
 \label{rank}
 Suppose that $k$ is not small w.r.t $X$.
 
 Then the following conditions are equivalent.
 
 \begin{itemize}
 \item[(1a)]
 There is a nontrivial admissible function on  $R_X$.
  \item[(1b)]
 There is a nontrivial admissible function on  $R_X$ that vanishes at $R_X^{\iota}$.
 \item[(2a)]
 There is a reduced  admissible function on  $R_X$.
  \item[(2b)]
 There is a reduced  admissible function on  $R_X$ that vanishes at $R_X^{\iota}$.
 \item[(3a)]
   $\rk(\Gamma(X,k)) \le \lfloor\#(R_X)/2\rfloor$.
    \item[(3b)]
   $\rk(\Gamma_1(X,k)) \le \lfloor\#(R_X)/2\rfloor$.
 \end{itemize}
 \end{lem}
 
 \begin{proof}
 Obviously, (1b) implies (1a),  (2b) implies (2a),   (2a) implies (1a), and (2b) implies (1b). By Lemma \ref{nonTrivRed}, (1a) implies (2b).
 This implies that (1a), (1b), (2a), (2b) are equivalent.
 
 In light of \eqref{rkGamma1},  conditions (3a) and  (3b) are   equivalent.
 %, since $\#(R_X\setminus R_X^{\iota})$ is even and
 %$$\#(R_X\setminus R_X^{\iota}) \le \#(R_X) \le \#(R_X\setminus R_X^{\iota})+1,$$
 %because $0\le \#( R_X^{\iota})\le 1$.
 
 In order to handle conditions (3), let us discuss the parity of $\#(R_X)$, using the observations and notation of Subsection \ref{involution}.
 
  In order to check the equivalence of (1) and (3), let us start with the ``degenerate'' case $r_X=0$, i.e., $R_X=R_X^{\iota}=\{\beta\}$.
 Then $\Gamma(X,k)$ is an infinite cyclic group generated by $\beta$ containing the index $2$ subgroup generated by $\beta^2=q$.
 Therefore   $\rk(\Gamma(X,k)) =1 > 0$, i.e., (3a) does {\sl not} hold. On the other hand, we have already seen (Lemma \ref{multiply})
 that if $R_X=\{\beta\}=R_X^{\iota}$ then (1a) does {\sl not} hold.
 
 So, we may  assume that $R_X \ne R_X^{\iota}$. Then the positive integer
 $r_X=\lfloor\#(R_X)/2\rfloor$ is the number of all $\iota$-orbits $O_1, \dots, O_{r_X}$ in $R_X \setminus R_X^{\iota}$, see Subsection \ref{Ror}.
 If we choose any element $\alpha_i$ of $O_i$ for all $i$ then the $2r_X$-element set
 $$R_X \setminus R_X^{\iota}=\{\alpha_1, q/\alpha_1, \dots, \alpha_{r_X},q/\alpha_{r_X}\}$$
 and $\Gamma_1(X,k)$ is generated by $q$ and $\{\alpha_1,\dots, \alpha_{r_X}\}$, see Subsection \ref{Ror}.
 
 Suppose that (3b)  holds.  This means that  $\rk(\Gamma_1(X,k))\le r_X$.
 Hence, there are $(r_X+1)$ integers $f_1, \dots, f_{r_X}; d$ {\sl not all} zeros, such that
 \begin{equation}
 \label{Delta}
 \prod_{i=1}^{r_X}\alpha_i^{f_i}=q^d.
 \end{equation} 
 Clearly, {\sl not all} $f_1, \dots, f_{r_X}$  are zeros.  Let us define the function
 \begin{equation}
 \label{Deltaf}
 e \colon R_X \to \Z, \ e(\alpha_i)=f_i \  \forall i=1, \dots, r_X; \ f(\alpha)=0 \ \text{ for all other } \alpha.
 \end{equation}
 In light of \eqref{Delta} and \eqref{Deltaf}, $e$ is a nontrivial  admissible function. Hence,
 (1a) holds. 
 
 Now assume that (1a) holds. Then (2b) holds, i.e., there are a {\sl nonempty} subset $A_1\subset R_X$, a function
 $\tilde{e} \colon A_1 \to \Z$ and a positive integer $s$ that enjoy the following properties.
 \begin{itemize}
 \item[(i)] $A_1$ and $\iota(A_1)$ do not meet each other;
 \item[(ii)] $ \tilde{e}(\alpha)>0 \ \forall \alpha\in A_1$;
 \item[(iii)] $\prod_{\alpha \in A_1}\alpha^{\tilde{e}(\alpha)}=q^s$.
 \end{itemize}
 Let us put $n:=\#(A_1)$ and let $A_1=\{\alpha_1, \dots \alpha_n\}$.  Then all
 $O_i=\{\alpha_i,q/\alpha_i\}$ are disjoint $2$-element orbits in $R_X\setminus R_X^{\iota}$. In particular, $n\le r_X$.
 
If $n=r_X$ then
 $\{\alpha_1, \dots, \alpha_{r_X}; q\}$ generate $\Gamma_1(X,k)$. The property (iii)
 implies that the rank of this group does not exceed $r_X$, i.e., (3b) holds.
 
 Now assume that $n<r_X$. Then there are  precisely $(r_X-n)$ {\sl other} two-element $\iota$-orbits $O_j$
 in $R_X$  ($j=n+1, \dots r_X$). If we pick for all $j$ an element $\delta_j \in O_j$ then $O_j=\{\delta_j, q/\delta_j\}$
 ($n+1 \le j \le r_X$). Then $\{\alpha_1, \dots, \alpha_n; \delta_{n+1}, \dots, \delta_{r_X};q\}$ generate a 
 subgroup of finite index in $\Gamma_1(X,k)$. The property (iii)
 implies that the rank of this group does not exceed $r_X$, i.e., (3b) holds.  This ends the proof.
 \end{proof}

 \section{Frames and Skeletons of Abelian Varieties over Finite Fields}
 \label{RelationProof}

 In the course of our proof of Theorem \ref{mainRelation} we will need the following notion.
 
 \begin{defn}
 Let $g$ be a positive integer. A $g$-frame is a triple $(M,r,U)$ that consists
 of positive integers $M$ and $r$, and a finite subset 
 $$U \subset \Q^M$$ of {\sl nonzero} vectors
 that enjoy the following properties.
 
 \begin{itemize}
 \item[(i)]
 $M$ divides $2^g g!$,  $r\le g$, and $\#(U)=2r$.
  \item[(ii)]
  $U \subset \Sl(g)^M\subset \Q^M$
  (see Lemma \ref{SLP} for the definition of the finite subset $\Sl(g)\subset \Q$).
  \item[(iii)] A vector $u\in \Q^M$ lies in $U$ if and only if
  $\mathbf{1}-u$ lies in $U$. Here
  $\mathbf{1}=(1, \dots,1)\in \Q^M$ is the vector, all whose coordinates are $1$.
  \item[(iv)]
  Let $\Delta(U)$ be the additive subgroup of $\Q^M$ generated by $\mathbf{1}$ and all elements of $U$.
  Then the rank of $\Delta(U)$  does not exceed $r$.
 \end{itemize}
 \end{defn}
 
 \begin{rem}
 \label{finiteF}
 The {\sl finiteness} of $\mathrm{Slp}(g)$ implies that the set of all frames (for a given $g$) is finite.
 \end{rem}
 
 \begin{sect}
 \label{invF}

The map
 \begin{equation}
 \label{iotaF}
 \iota_F \colon \Q^M \to \Q^M, \ u \mapsto \mathbf{1}-u
 \end{equation}
 is an involution, whose only fixed point is
 $$\frac{1}{2}\cdot \mathbf{1}=(1/2, \dots,1/2).$$
 Notice that 
 $$\iota_F(U)=U.$$
 Since $\#(U)$ is {\bf even}, $U$  does {\sl not} contain the {\sl fixed point} $\frac{1}{2}\cdot \mathbf{1}$
 and therefore splits
   into a disjoint union of $2$-element $\iota_F$-orbits  $O_1, \dots, O_r$.
 If we choose in each $O_i$ a vector $u_i \in O_i$ then
 \begin{equation}
 \label{UrOrbits}
 O_i=\{u_i,\mathbf{1}-u_i\} \ \forall i=1, \dots, r; \ U=\{u_1, ,\mathbf{1}-u_1, \dots, u_i, \mathbf{1}-u_i, \dots, \mathbf{1}-u_r\}
 \end{equation}
 Property (iv) combined with \eqref{UrOrbits} implies that there  exist integers $a_1, \dots, a_r$ not all zeros and an integer $d$
 such that
 \begin{equation}
 \label{relationF}
 \sum_{i=1}^r a_i u_i=d \cdot \mathbf{1}=(d,\dots, d).
 \end{equation}
 \end{sect}

\begin{lem}
\label{boundU}
Let $g$ be a positive integer. Then there is a positive integer $C(g)$ that depends only on $g$ and enjoys the following property.

 Let  $(M,r,U)$ be a $g$-frame. Then there are exist $r$ integers $a_1, \dots, a_r$ not all zeros, an integer $d$,
 and $r$ distinct vectors $u_1, \dots ,u_r$ in $U$ such that:
 
 \begin{itemize}
 \item[(i)]
  the $2r$-element set
$U=\{u_1, \mathbf{1}-u_1, \dots, u_i,\mathbf{1}-u_r\}$;
\item[(ii)] $\sum_{i=1}^r a_i u_i=d \cdot \mathbf{1}=(d,\dots, d)$;
\item[(ii)] $\sum_{i=1}^r |a_i| \le C(g)$.
\end{itemize}
\end{lem}

\begin{proof}
The assertions follow readily from the construction of Subsection \ref{invF} combined with Remark \ref{finiteF}.

\end{proof}
 
 \begin{sect}
 \label{skeleton}
 Let $X$ be a $g$-dimensional abelian variety over a finite field $k$ of characteristic $p$.
 Suppose that $k$  is {\sl not small} with respect to $X$ and there exists a nontrivial admissible function $R_X \to \Z$. The aim of this subsection is to assign to $X$
 a certain $g$-frame that we call the {\sl skeleton} of $X$.
 
  First, let us put $r:=r_X$ and
 $M:=M_X:=\#(S(p))$ where $S(p)$ is the set of maximal ideals in $\Oc_{L_X}$ that lie above $p$ (see Remark \ref{ordP}).
It follows from Lemma \ref{basicL} that $M$ divides $2^g\cdot g!$.
 By Lemma \ref{multiply}, the existence of a nontrivial admissible function implies that $R_X \ne R_X^{\iota}$ and $r=r_X$ is a positive integer.
 In addition  (see \eqref{rXdef}),
 $$r \le g, \ 2r=\#(R_X\setminus R_X^{\iota}).$$
 Let us choose an order on the $M$-element set $S(p)$. This allows us to identify $S(p)$ with $\{1, \dots, M\}$ and $\Q^{S(p)}$ with $\Q^M$.
 Let us put
 $$U=U_X:=w_X(R_X\setminus R_X^{\iota})\subset \Q^{S(p)}=\Q^M$$
 (where homomorphism $w_X$ is defined in Remark  \ref{ordP}).
 It follows from Remark  \ref{ordP}(d) that the map
 \begin{equation}
 \label{inRX}
 R_X\setminus R_X^{\iota} \to U_X,  \ \alpha \mapsto w_X(\alpha)
 \end{equation}
 is {\sl injective};  in particular,
 $$2r=2 r_X=\#(R_X\setminus R_X^{\iota})=\#(U_X).$$
 Since $\ker(w_X)$ consists of roots of unity (see Remark  \ref{ordP}(d)), the rank of $\Delta(U_X)$ coincides
 with the rank of multiplicative $\Gamma_1(X,k)$ generated by $R_X\setminus R_X^{\iota}$. The existence of a nontrivial
 admissible function implies (thanks to Theorem \ref{rank}) that
 \begin{equation}
 \label{rankU}
 \rk(\Delta(U_X))=\rk(\Gamma_1(X,k)) \le r_X.
 \end{equation}

I claim that $(M_X,r_X,U_X)$ is a $g$-frame. Indeed,  it follows from Remarks \ref{ordP} that
\begin{equation}
\label{ordPrevisited}
w_X(q)=\mathbf{1};  w_X(\alpha)\ne 0, \  w_X(q/\alpha)=\mathbf{1}-w_x(\alpha) \ \forall \alpha \in R_X\setminus R_X^{\iota}.
\end{equation}
 This implies that $(M_X,r_X,U_X)$ enjoys the properties (i)-(iii).  
 %It follows from Remark \ref{ordP}(e) combined with   Remark \ref{ssF}  that (iv) also holds. 
  As for (iv), its validity follows from \eqref{rankU}.
\end{sect}

\begin{proof}[Proof of Theorem \ref{mainRelation}]
Let $g$ be a positive integer.
In light of Lemma \ref{rootsR}, we may and will assume that $k$ is {\sl not small} w.r.t. $X$.
In light of Lemma \ref{nonTrivRed}, it suffices to prove the following assertion.

{\bf Claim}. {\sl  There exists a positive integer $E(g)$ that depends only on $g$ and enjoys the following property.
Suppose that $X$ is a $g$-dimensional abelian variety over a finite field $k$ such that $k$ is not small w.r.t. $X$ and there exists a nontrivial admissible function $R_X \to \Z$.

Then there exists a nontrivial admissible function $R_X \to \Z$
of weight $\le E(g)$.}

\begin{proof}[Proof of Claim]
Let $X$ be an $g$-dimensional abelian variety over a finite field $k$ such that $k$ is not small w.r.t. $X$ and there exists a nontrivial admissible function $R_X \to \Z$.
Let us consider the corresponding $g$-frame $(M_X,r_X,U_X)$. It follows from the injectiveness of the map \eqref{inRX} combined with Lemma \ref{boundU}
that there exist  $r_X$ distinct elements $\alpha_1, \dots, \alpha_{r_X} \in R_X\setminus R_X^{\iota}$, $r_X$ integers $a_1, \dots, a_{r_X}$, and an integer $d$
that enjoys the following properties.

\begin{enumerate}
\item[(1)]
 $R_X\setminus R_X^{\iota}=\{\alpha_1,q/\alpha_1 \dots, \alpha_{r_X},q/\alpha_{r_X}\}$. 
\item[(2)]
Not all $a_1, \dots a_{r_X}$ are zero.
\item[(3)]
$\sum_{i=1}^{r_X} a_i w_X(\alpha_i)=d \cdot \mathbf{1}=(d,\dots, d)$.
\item[(4)] $\sum_{i=1}^r |a_i| \le C(g)$.
(Here $C(g)$ is as in Lemma  \ref{boundU}.) 
\end{enumerate}

It follows from Remark \ref{ordP}(d) that there exists
a root of unity $\gamma \in \Gamma(X,k)$ such that 
$$\prod_{i=1}^{r_X}\alpha_i^{a_i}=q^d \gamma.$$
According to Lemma \ref{rootsG}, there exists a positive integer
$m \le D(g)$ such that $\gamma^m=1$. (See Lemma \ref{rootsG}
for the explicit formula of $D(g)$.) This implies that
$$\prod_{i=1}^{r_X}\alpha_i^{ma_i}=q^{md}.$$
This implies that the function
$$e \colon R_X \to \Z, \ e(\alpha_i)=m\cdot a_i \ \forall \alpha_i, \ e(\alpha)=0 \ \text{ forall other } \alpha$$
is admissible. On the other hand, it follows from properties (1) and (2) that $e$ is {\sl nontrivial}.
In order to finish the proof of Claim, one has only  to notice
that
$$\wt(e)=\sum_{i=1}^{r_X}|a_i|=m \sum_{i=1}^{r_X}|a_i| \le D(g)\cdot C(g)=:E(g).$$

\end{proof}

This ends the proof of Theorem \ref{mainRelation}.

\end{proof}

\section{Applications of Gordan's Lemma}
\label{GordanL}

In order to prove Theorem \ref{semigroupADM}, we need the following 
variant of a classical result of P. Gordan.

\begin{lem}
\label{GordanVar}
Let $m$  and $s$ be  positive integers and $v_1, \dots v_s$ be elements of $\Q^m$. Let us consider the addditive semigroup
$$W=\{u \in \Z_{+}^m \mid u \cdot v_j=0 \ \forall j=1, \dots, s\}\subset \Z_{+}^m.$$
Then $W$ is a finitely generated semigroup of $\Z_{+}^m$.
\end{lem}

\begin{proof}
Replacing all $v_j$ by $N v_j$, where $N$ is a sufficiently divisible positive integer, we may and will assume that $v_j \in \Z^m$ for all $j=1, \dots, s$.

Let us consider  the rational polyhedral cone
$\sigma\subset \RR^m$ that is generated by the standard  basis of $\RR^m$ and all the vectors $\{v_1, \dots v_s\}$. Then
the dual cone is
$$\sigma^{\vee}=\{u\in \RR_{+}^m\mid u \cdot v_j\ge 0 \ \forall j=1, \dots, s\}.$$
By Gordan's Lemma  \cite[Ch. 1, Prop. 1.2.17]{Cox}, $\sigma^{\vee}\cap \Z^m$ is a finitely generated additive semigroup. Let $G$ be  a  finite subset of $\sigma^{\vee}\cap \Z^m$ 
 that contains $0$
and generates  $\sigma^{\vee}\cap \Z^m$.
Then the intersection  $G\cap W$ is a finite subset of $W$ that contains $0$. I claim that $G\cap W$ generates $W$ as a semigroup. Indeed, if $w \in W$ then $w \in \sigma^{\vee}\cap \Z^m$ and therefore  there exists a positive integer $r$ and  (not necessarily distinct) $r$ elements $g_1, \dots g_r \in G$ such that
$w=\sum_{i=1}^r g_i$. We have for all $j=1, \dots, s$
$$0=w  \cdot v_j=\sum_{i=1}^r g_i\cdot v_j, \ g_i\cdot v_j \ge 0  \ \forall i=1, \dots, s.$$
This implies that all $g_i\cdot v_j=0$ and therefore all $g_i \in W$, i.e., $g_i \in G\cap W$.
It follows that  $G\cap W$ generates $W$ as a semigroup. 
\end{proof}

We also need the following elementary observation.

\begin{lem}
\label{F2q}
Suppose that $X$ is an abelian variety of positive dimension $g$ over a finite field $k$ with $q$ elements.
%of characteristic $p$.
 Suppose that $k$ is sufficiently large w.r.t. $X$.
Then a nonnegative integer-valued function $e \colon R_X \to \Z_{+}$ of even weight is admissible if and only if
\begin{equation}
\label{normalizedF}
w_X\left(\prod_{\alpha\in R_X}(\alpha^2/q)^{e(\alpha)}\right)=0 \in \Q^{S(p)}.
\end{equation}
\end{lem}

\begin{rem}
\label{AdmEven}
Let $e \colon R_X \to \Z_{+}$ be an admissible  nonnegative integer-valued function.
Then its weight is twice its degree and therefore is {\sl even}. 
\end{rem}

\begin{proof}[Proof of Lemma \ref{F2q}]
Since $e$ is nonnegative, its weight  coincides with
$$\sum_{\alpha\in R_X} e(\alpha)=:n.$$
 Since this weight  is even, 
there is a nonnegative integer $d$ such that $n=2d$.

Now notice that in light of Remark \ref{ordP}(d), 
\eqref{normalizedF}
 holds if and only if $\prod_{\alpha\in R_X}(\alpha^2/q)^{e(\alpha)}$ is a root of unity.
This means that $\prod_{\alpha\in R_X}(\alpha^2/q)^{e(\alpha)}=1$, because $k$ is sufficiently large w.r.t. $X$. Hence,
\eqref{normalizedF}  means that
$$\left(\prod_{\alpha\in R_X}\alpha^{e(\alpha)}\right)^2=q^n \ \text{ with } n=\sum_{\alpha\in R_X} e(\alpha)=2d.$$
This means that
\begin{equation}
\label{pmqd}
\prod_{\alpha\in R_X}\alpha^{e(\alpha)}=\pm q^d.
\end{equation}
Since torsion-free $\Gamma(X,k)$  does not contain $-1$, \eqref{pmqd} is equivalent to
$$\prod_{\alpha\in R_X}\alpha^{e(\alpha)}= q^d,$$
i.e, $e$ is admissible.
\end{proof}

\begin{proof}[Proof of Theorem \ref{semigroupADM}]
Let $X$ be an abelian variety of positive dimension $g$ over a finite field $k$ of characteristic $p$. Suppose that $k$ is sufficiently large w.r.t. $X$.
Let us put $s:=\#(S(p))$. By Lemma \ref{basicL}, $s$ divides $2^g \cdot g!$. Let us choose an order in $S(p)$. This allows us to identify $S(p)$ with $\{1, \dots,s\}$
and $ \Q^{S(p)}$ with $\Q^s$.
Let us choose an order on $R_X$: it allows us to list elements of  $R_X$ as  $\{\alpha_1, \dots, \alpha_m\}$ with $m=\#(R_X)$. We have $m \le 2g$. Let us consider an additive group homomorpism
$$\tilde{w}_X \colon \Z^m \to \Q^{S(p)}=\Q^s, \ u=(a_1, \dots a_m)\mapsto w_X\left(\prod_{i=1}^m (\alpha_i^2/q)^{a_i}\right)=$$
$$2\sum_{i=1}^m a_i w_X(\alpha_i)-\left(\sum_{i=1}^m a_i\right)\cdot \mathbf{1}.$$
%By  Lemma \ref{F2q}, a vector $u=(a_1, \dots a_m)\in \Z_{+}^m$ lies in the kernel of $\tilde{w}_X$ if and only if the nonnegative integer-valued function
%$$e_u: R_X \to \Z_{+}, \ \alpha_i \mapsto a_i$$
%is {\sl admissible}.
Clearly, there is a unique collection of $s$ vectors $v_1, \dots v_s \in \Q^m$ such that
$$\tilde{w}_X(u)=(u\cdot v_1, \dots , u\cdot v_s) \ \forall u \in \Z^m.$$
It is also clear that all the coordinates of all $v_j$'s lie in the same finite set 
$$2\cdot S(g)-1:=\{2c-1\mid c \in S(g)\} \subset \Q$$
 that depends only on $g$.  This implies that all the $v_j$'s lie in the same finite subset 
  $$\left(2\cdot S(g)-1\right)^m \subset    \Q^m$$ 
 of $\Q^m$  that depends only on $g$ and $m$.  Combining this assertion with Lemma \ref{GordanVar}, we obtain 
  that for each positive integers $m \le 2g$ and $s$ dividing  $2^g \cdot g!$ there is a {\sl finite subset}
$F_0(g,m,s) \in \Z_+^{m}$ that depends only on $g$, $m$, $s$ and enjoys the following property.

{\sl If $\#(R_X)=m$ and $\#(S(p))=s$ then the additive semigroup $\ker(\tilde{w}_X)\cap \Z_{+}^m$ of $\Z_{+}^m$ is generated by a certain subset of  $F_0(g,m,s)$.}

Now let as define the weight $\wt(u)$ of any $u=(a_1, \dots, a_m) \in \Z_{+}^m$ as $\sum_{i=1}^m a_i$. It follows from Remark \ref{AdmEven} combined with Lemma \ref{F2q}
that an integer-valued nonnegative function
$$\mathbf{b}_u \colon R_X \to \Z_{+}, \  \alpha_i \mapsto a_i$$
is {\sl admissible} if and only if $u \in \ker(\tilde{w}_X)\cap \Z_{+}^m$ and $\wt(u)$ is {\sl even}. (It is also clear that each admissible nonnegative function $e: R_X \to \Z_{+}$
coincides with $\mathbf{b}_u$ for exactly one vector $u\in \Z_{+}^m$.) Then such $u$ may be presented as a sum of (not necessarily distinct) elements of $F_0(g,m,s)$. It may happen that
some elements of $F_0(g,m,s)$ in this sum have odd weight. Since the weight of $u$ is even, the number of such summands is even. By grouping them in pairs, we obtain
that $u$ is a finite sum of some even weight elements from $F_0(g,m,s)$ and  even weight elements from $F_0(g,m,s)+F_0(g,m,s)\subset \Z_{+}^m$. Now let $F(g,m,s) \subset \Z_{+}^m$ be 
the (finite) set of all even weight vectors from $F_0(g,m,s)$ and from 
$F_0(g,m,s)+F_0(g,m,s)$. Clearly, each $\mathbf{u} \in  F(g,m,s)$ gives rise to nonnegative admissible
$\mathbf{b}_{\mathbf{u}} \colon R_X \to \Z_{+}$ and each nonnegative admissible $e \colon R_X \to \Z_{+}$ may be presented as a linear combination of such $\mathbf{b}_{\mathbf{u}}$'s with nonnegative integer coefficients.
Now one only has to choose as $H(g)$ the largest of  the weights of $\mathbf{b}$ among all $\mathbf{u}$ (with {\sl even weight})
in the union of all
 $F(g,m,s)$ where $1 \le m \le 2g$ and  $s\mid 2^g \cdot g!$.
\end{proof}

 Theorem \ref{semigroupADM} implies readily the following assertion.
 
 \begin{cor}
 \label{contractionE}
 Let $g$ be a positive integer and $H(g)$ be as in  Theorem \ref{semigroupADM}.
 
 Let  $X$ an abelian variety of positive dimension $g$ over a finite field $k$. Let $e: R_X \to \Z_{+}$
 be a nonnegative admissible function. If $\wt(e)>H(g)$ then $e$ may be presented as a sum 
 $e=f_1+f_2$ of two nonnegative admissible functions
 $$f_1 \colon R_X \to \Z_{+}, \ f_2 \colon R_X \to \Z_{+}$$
 such that
 $2 \le \wt(f_2)\le H(g)$.
 \end{cor}

\section{Linear Algebra}
\label{linAlgebra}

Throughout this section,  $V$ is a nonzero vector space of finite dimension $n$ over a field $K$ of characteristic $0$, and $E$ is an overfield of $K$. We write
$V_E$ for the $E$-vector space $V\otimes_K E$ of the $E$-dimension $n$. Let us put
$$V^{*}=\Hom_K(V,K),  \ V_E^{*}=\Hom_E(V_E,E).$$
Let $A \colon V \to V$ be a $K$-linear operator and
$$A^{*} \colon V^{*} \to V^{*},  \  \phi \mapsto \phi\circ A \ \forall \phi\in  V^{*}.$$
As usual, let us define
 $$A_E \in \End( V_E),   \  A_E (v\otimes e)= Av\otimes e  \ \forall v \in V, e \in E.$$
 Clearly,
 $$(A_E)^{*}=(A^{*})_E \colon V_E^{*} \to  V_E^{*}.$$
 \begin{rem}
 \label{eigenKE}
 Let $a \in K\subset E$ and $V(a)$ (resp. $V_E(a)$) be the eigenspace of $A$ (resp. of $A_E$) attached to eigenvalue $a$.
It is well known that the natural $E$-linear map
 $$V(a)\otimes_K E \to V_E(a)$$ is an isomorphism of $E$-vector spaces; in particular,
 $$\dim_K(V(a))=\dim_E( V_E(a)) \ \forall a \in K\subset E.$$
 \end{rem}
There are well known  natural isomorphisms  \cite[Ch. III, Sect. 7, Prop. 7]{Bourbaki} of graded $K$-algebras
$$\wedge(V^*)=\oplus_{j=0}^{n}\wedge_K^{j} (V^*)=\oplus_{j=0}^{n}\Hom_K(\wedge_K^j (V),K)$$
 and of graded $E$-algebras \cite[Ch. III, Sect. 7, Prop. 8]{Bourbaki} 
$$\wedge(V_E^*)=\oplus_{j=0}^{n} \wedge_E^j (V_E^*)=\oplus_{j=0}^{n}\Hom_E(\wedge_E^j (V_E),E)=\oplus_{j=0}^{n}\Hom_K(\wedge_K^j (V),K)\otimes_K E,$$
which give rise to the natural isomorphisms of $E$-vector spaces
\begin{equation}
\label{baseEKchange}
\wedge_K^{j} (V^*)_E \cong  \wedge_E^j (V_E^*).
\end{equation}

\begin{sect}
\label{wedgeImage}
Let $i$ and $j$ be nonnegative integers.
The multiplication in $\wedge(V^*)$ (resp. in $\wedge(V_E^*)$) gives rise to the surjective $K$-linear map
\begin{equation}
\label{wedgeKsurGeneral}
\Lambda_{i,j,K} \colon \wedge_K^{i} (V^*)\otimes_K \wedge_K^{j} (V^*)\twoheadrightarrow \wedge_K^{i+j} (V^*), \ \psi_i\otimes \psi_j \mapsto \psi_i\wedge \psi_j 
\end{equation}
and to the surjective $E$-linear map
\begin{equation}
\label{wedgeEsurGeneral}
\Lambda_{i,j,E} \colon \wedge_E^{i} (V_E^*)\otimes_E \wedge_E^{j} (V_E^*)\twoheadrightarrow \wedge_E^{i+j} (V_E^*), \ \psi_i\otimes \psi_j \mapsto \psi_i\wedge \psi_j. 
\end{equation}
Let $U$ be a $K$-vector subspace in $\wedge_K^{i} (V^*)$ and  $W$ be a $K$-vector subspace in $\wedge_K^{j} (V^*)$. Then obviously the images
$\Lambda_{i,j,K}(U\otimes_K W)\subset \wedge_K^{i+j} (V^*)$ and $\Lambda_{i,j,E}(U_E\otimes_E W_E)\subset \wedge_E^{i+j} (V_E^*)$ are related by
\begin{equation}
\label{wedgeUVKE}
\Lambda_{i,j,E}(U_E\otimes_E W_E)=\Lambda_{i,j,K}(U\otimes_K W)_E.
\end{equation}
Here we identify $U_E$ (resp. $W_E$) with its isomorphic image in $\wedge_K^{i} (V^*)_E=\wedge_E^{i} (V_E^*)$
(resp. in  $\wedge_K^{j} (V^*)_E=\wedge_E^{j} (V_E^*)$).

The equality \eqref{wedgeUVKE} implies readily its own generalization. Namely,  let $n$ be a positive integer and suppose that for each positive integer $r\le n$
we are given $K$-vector subspaces
$$U_r \subset \wedge_K^{i} (V^*), \ W_r \subset \wedge_K^{j} (V^*).$$
Then
\begin{equation}
\label{wedgeUWrKE}
\sum_{r=1}^n\Lambda_{i,j,E}(U_{r,E}\otimes_E W_{r,E}) =\left(\sum_{r=1}^n \Lambda_{i,j,K}(U_r\otimes_K W_r)\right)_E.
\end{equation}
Here  
$$U_{r,E}=U_r\otimes_K E,  \ W_{r,E}=W_r\otimes_K E.$$
\end{sect}

\begin{sect}
The operators $A^{*}$ and $A_E^{*}$ give rise to the  graded $K$-algebra and graded $E$-algebra {\sl endomorphisms} 
%\begin{equation}
%\label{wedgeAlgebra}
$$\wedge(A^*) \colon \wedge(V^*) \to \wedge(V^*), \ \wedge(A_E^*): \wedge(V_E^*) \to \wedge(V_E^*)$$
%\end{equation}
\cite[Ch. III, Sect. 7, Prop. 2]{Bourbaki}, whose homogeneous components are
$K$-linear and $E$-linear operators
%\begin{equation}
%\label{wedgeAlgebraJ}
$$\wedge^j(A^{*})\colon \wedge_K^{j} (V^*) \to \wedge_K^{j} (V^*), \ \wedge^j_E(A_E^{*}):\wedge_E^{j} (V_E^*) \to \wedge_E^{j} (V_E^*)$$
%\end{equation}
respectively, 
such that
\begin{equation}
\label{WEK}
\wedge^j(A_E^{*})=\wedge^j (A^{*})_E.
\end{equation}
Since $\wedge(A)$ and $\wedge(A_E^*)$ respect the multiplication in $\wedge(V^*)$ and $ \wedge(V_E^*) $ respectively,
\begin{equation}
\label{multipliWedge}
\wedge^i(A^{*})(\psi_i)\wedge \ \wedge^j(A^{*})(\psi_j)=\wedge^{i+j}(A^{*})(\psi_i\wedge \psi_j)\in  \wedge^{i+j} (V^{*}) \ \forall \psi_i \in \wedge^i (V^{*}), \psi_j \in \wedge^j (V^{*});
\end{equation}
$$\wedge^i(A_E^{*})(\psi_{i,E})\wedge \ \wedge^j(A_E^{*})(\psi_{j,E})=\wedge^{i+j}(A_E^{*})(\psi_{i,E}\wedge \psi_{j,E}) \in  \wedge^{i+j} (V_E^{*})$$ 
 $$ \forall \psi_{i,E} \in \wedge^i (V_E^{*}), 
\psi_{j,E} \in \wedge^j (V_E^{*}).$$
The following assertion is an immediate corollary of \eqref{multipliWedge} and \eqref{WEK}.
\end{sect}

\begin{lem}
\label{product EigenSpaces}
Let $j_1, j_2$ be positive integers such that $j_1+j_2 \le \dim(V)$.  Let $\lambda_1,\lambda_2$ be elements of $K$. Let
$\wedge^{j_r}_K (V^{*})(\lambda_r)\subset \wedge^{j_r}_K (V^{*})$ be the eigenspace of $\wedge^j_r(A^{*})$ attached to  $\lambda_r$ ($r=1,2$).
Then the image of the $K$-linear map 
$$\wedge^{j_1}_K (V^{*})(\lambda_1)\otimes_K \wedge_K^{j_2} (V^{*})(\lambda_2) \to \wedge_K^{j_1+j_2} (V^{*}),  \ \psi_{j_1}\otimes\psi_{j_2} \mapsto \psi_{j_1}\wedge\psi_{j_2}$$
lies in the eigenspace $\wedge^{j_1+j_2}_K (V^{*})(\lambda_1\lambda_2)$ of  $\wedge^{j_1+j_2}(A^{*})$ attached to  $\lambda_1 \lambda_2$.
\end{lem}

\begin{rem}
\label{restrictionLambdaK}
The $K$-linear map in Lemma \ref{product EigenSpaces}
 is the restriction of $\Lambda_{j_1,j_2,K}$ defined in Subsection \ref{wedgeImage}.
\end{rem}

\begin{rem}
\label{eigenKEwedge}
Applying Remark \ref{eigenKE} to $\wedge^j (A^{*}) \colon \wedge_K^{j} (V^*) \to \wedge_K^{j} (V^*)$ (instead of $A \colon V \to V$), we obtain that if  $\lambda \in K\subset E$
and $\wedge_K^{j} (V^*)(\lambda)$ (resp. $\wedge^j(V_E^{*})(\lambda)$) is  the attached to $\lambda$  eigenspace of $\wedge^j (V^{*})$ (resp. of $\wedge^j(V_E^{*})$)
then the natural $E$-linear map
$$\wedge_K^{j} (V^*)(\lambda)\otimes_K E \to \wedge^j(V_E^{*})(\lambda)$$
induced by \eqref{baseEKchange} is an isomorphism. Combining this assertion with Lemma \ref{product EigenSpaces} applied twice (over $K$ and over $E$),
we get immediately the following assertion.
\end{rem}

\begin{lem}
\label{WedgeEigen} Let $j_1, j_2$ be positive integers such that $j_1+j_2 \le \dim(V)$.
Let $\lambda_1,\lambda_2 \in K \subset E$.
We keep the notation and assumptions of Lemma \ref{product EigenSpaces}. 

The $E$-linear map
$$\wedge^{j_1}_E (V_E^{*})(\lambda_1)\otimes_E \wedge^{j_2}_E (V_E^{*})(\lambda_2) \to \wedge_E^{j_1+j_2} (V^{*})(\lambda_1\lambda_2),  \ \psi_1\otimes\psi_1 \mapsto \psi_1 \wedge\psi_2$$
is not surjective if and only if the $K$-linear map
$$\wedge^{j_1}_K(V^{*})(\lambda_1)\otimes_K \wedge^{j_2}_K (V^{*})(\lambda_2) \to \wedge^{j_1+j_2}_K (V^{*})(\lambda_1\lambda_2),  \ \psi_1\otimes\psi_1 \mapsto \psi_1 \wedge\psi_2$$
is not surjective.   Here $\wedge^{j}_E (V_E^{*})(\lambda)\subset \wedge^{j}_E (V_E^{*})$ is the eigensubspace of $\wedge^j(A_E^{*})$ attached to $\lambda$.
\end{lem}

\begin{sect}{\bf Main construction.}
\label{multEigen}
We keep the notation of Remark \ref{eigenKEwedge}. Suppose that $A_E \colon  V_E \to V_E$ is {\sl diagonalizable},
$\mathrm{spec}(A)\subset E$ is the set of its {\sl eigenvalues}, and $\mult_A \colon \mathrm{spec}(A) \to \Z_{+}$ is the integer-valued function
that assigns to each eigenvalue of $A_E$ its multiplicity.  

Let $\lambda \in K$ and  $j \le \dim(V)$ be a positive integer.
%the number $N(j,a;A)$ of
Let us consider an integer-valued
function $e \colon \mathrm{spec}(A) \to \Z_{+}$ that enjoys the following properties.

\begin{itemize}
\item[(i)] $e(\alpha)\le \mult_A(\alpha) \ \forall \alpha \in \mathrm{spec}(A)$;

\item[(ii)] $\sum_{\alpha\in  \mathrm{spec}(A)} e(\alpha)=j$;

\item[(iii)] $\prod_{\alpha\in  \mathrm{spec}(A)}\alpha^{ e(\alpha)}=\lambda$.
\end{itemize}

Let us choose  an {\sl eigenbasis}  $B$ of $E$-vector space $V_E$ w.r.t. $A_E$ and let  
$$\pi \colon B \twoheadrightarrow \mathrm{spec}(A)$$ be
the surjective map that assigns to each eigenvector $x \in B$ the corresponding eigenvalue of $A_E$. Clearly,
for every eigenvalue $\alpha\in  \mathrm{spec}(A)$ the preimage $\pi^{-1}(\alpha)$ consists of $\mult_A(\alpha)$ elements of $B$. 
Let $$B^{*}=\{x^{*}\mid x\in B\}$$ be the basis of $V_E^{*}$ that is dual to $B$.
Let us choose an order on $B$ and define for each $j$-element subset $C\subset B$ an element
$$y_C:=\wedge_{x \in C} x^{*}\in \wedge^j(V_E^{*}).$$
Clearly, all $y_C$'s constitute an {\sl eigenbasis} of $\wedge^j(V_E^{*})$ w.r.t. $\wedge^j (A^{*})$. Actually,
%\begin{equation}
%\label{wedgeEigen}
$$\wedge^j (A^{*}) (y_C)= \left(\prod_{x\in C}\pi(x)\right) y_C.$$
%\end{equation}
Let us assign to $C$ the integer-valued function
\begin{equation}
\label{eCdef}
e=e_C \colon \mathrm{spec}(A) \to \Z_{+},   \ \alpha \mapsto \#(\{x \in C \mid \pi(x)=\alpha\}).
\end{equation}
Clearly, $y_C$ is an eigenvector of $\wedge^j (A^{*})$ with eigenvalue $\lambda$ if and only if $e_C$ enjoys the properties (i)-(iii).
This implies that the set of $y_C$'s such that $e_C$ satisfies (i)-(iii) is a $E$-basis of the eigenspace $\wedge^j(V_E^{*})(\lambda)$.

Conversely, suppose that $e \colon \mathrm{spec}(A) \to \Z_{+}$ is an integer-valued function that enjoys the properties (i)-(iii). I claim that
there exists a $j$-element subset $C\subset B$ such that $e=e_C$. Indeed,  let us choose a $e(\alpha)$-element subset
$C_{\alpha}\subset \pi^{-1}(\alpha)\subset B$ for all $\alpha \in \mathrm{spec}(A)$ with $e(\alpha)>0$. The property (i) guarantees that such a choice is possible
(but not necessarily unique). Now define $C$ as the (disjoint) union of all these $C_{\alpha}$'s. Property (ii) implies that $B$ is a $j$-element subset of $B$.
It follows from (iii) that $y_C \in \wedge^j(V_E^{*})(\lambda)$.

\end{sect}

The following assertion will be used in the proof of Theorem \ref{mainTate} (with $K=\Q_{\ell}, V=V_{\ell}(X^n), A=\Fr_{X^n}$).

\begin{prop}
\label{exoticKE}
We keep the notation and assumptions  of Subsection \ref{multEigen}, Remark \ref{eigenKE} and Lemma \ref{WedgeEigen}. In particular, $A_E \colon V_E \to V_E$
is diagonalizable.
Assume additionally that $A \colon V \to V$ is invertible,  $j_1=j-2$ and $j_2=2$.
Suppose that $\lambda_1$ and $\lambda_2$ are nonzero elements of $K$ and $j>2$, i.e., $j_1\ge 1$.
Then the following conditions are equivalent.
\begin{itemize}
\item[(a)]  The $K$-linear map
\begin{equation}
\label{imageWDGK}
\wedge_K^{j-2}(V^{*})(\lambda_1) \otimes_K \wedge_K^2(V^{*})(\lambda_2) \to \wedge_K^j(V^{*})(\lambda_1\lambda_2),  \psi\otimes \phi \mapsto \psi\wedge \phi
\end{equation}
is not surjective. 
\item[(b)]
There exists a function $e \colon \mathrm{spec}(A) \to \Z_{+}$ that enjoys the  following properties.

\begin{itemize}
\item[(i)] $e(\alpha)\le \mult_A(\alpha) \ \forall \alpha \in \mathrm{spec}(A)$;

\item[(ii)] $\sum_{\alpha\in  \mathrm{spec}(A)} e(\alpha)=j$;

\item[(iii)] $\prod_{\alpha\in  \mathrm{spec}(A)}\alpha^{ e(\alpha)}=\lambda_1\lambda_2$.

\item[(iv)]
If $\alpha \in \mathrm{spec}(A)$ and 
%either $e(\alpha) =0$ or 
$e(\alpha) \ne 0$ then 
$e(\alpha) \ge 1$ and one of the following conditions holds.

\begin{enumerate}
\item[(1)] $\lambda_2/\alpha \not\in  \mathrm{spec}(A)$;
\item[(2)]  $\lambda_2/\alpha \in  \mathrm{spec}(A)$ but $e(\lambda_2/\alpha)=0$.
\item[(3)]  $\alpha=\lambda_2/\alpha$ (i.e., $\alpha^2=\lambda_2$) and $e(\alpha)=1$.
\end{enumerate}
\end{itemize}
\end{itemize}

\end{prop}

\begin{rem}
The invertibility of $A$ means that $0\not\in \mathrm{spec}(A)$.
\end{rem}

\begin{rem}
\label{KESUR}
In light  of Lemma \ref{WedgeEigen}, it suffices to check that condition (b) is equivalent (in the obvious notation) to the {\sl non-surjectiveness} of  the $E$-linear map 
\begin{equation}
\label{imageWDG}
\wedge_E^{j-2}(V_E^{*})(\lambda_1) \otimes_E \wedge_E^2(V_E^{*})(\lambda_2) \to \wedge^j(V^{*})(\lambda_1\lambda_2),  \psi\otimes \phi \mapsto \psi\wedge \phi.
\end{equation}
\end{rem}

\begin{proof}[Proof of Proposition \ref{exoticKE}]
We  start with the following lemma that describes the image of map \eqref{imageWDG}.

\begin{lem}
\label{imageWedgePr}
The image of the map \eqref{imageWDG} is generated by all $y_C$'s where $C$ is any $j$-element subset of $B$ that enjoys the following properties.

The set $C$ is a disjoint union of a $(j-2)$-element subset $S$ and a $2$-element subset $T$ such that the corresponding functions
$$e_S \colon \mathrm{spec}(A) \to \Z_{+}, \quad e_T \colon \mathrm{spec}(A) \to \Z_{+}$$
defined as in  \eqref{eCdef} enjoy the following properties.
\begin{equation}
\label{STC}
\prod_{\alpha\in  \mathrm{spec}(A)}\alpha^{ e_S(\alpha)}=\lambda_1, \ \prod_{\alpha\in  \mathrm{spec}(A)}\alpha^{ e_T(\alpha)}=\lambda_2.
\end{equation}
\end{lem}

\begin{proof}[Proof of Lemma \ref{imageWedgePr}]
It follows from arguments of Subsection \ref{multEigen} that all the $y_S$'s  (resp. all the $y_T$'s)  where $S$ is any $(j-2)$-element subset of $B$ 
(resp. where $T$ is any $2$-element subset of $B$)  that satisfies \eqref{STC} constitute a basis of $\wedge_E^{j-2}(V_E^{*})(\lambda_1)$
(resp. a basis of $\wedge_E^2(V_E^{*})(\lambda_2)$). This implies that the image of map \eqref{imageWDG} is generated by all $y_S\wedge y_T$.
If $S$ meets $T$ then it follows from the
very definition of $y_S$ and $y_T$  and basic properties of wedge products that 
$y_S \wedge y_T=0$. On the other hand, if $S$ does {\sl not} meet $T$ then $C:=S\cup T=S \sqcup T$ is a $j$-element subset of $B$ and $y_S \wedge y_T=\pm y_{C}$.
This ends the proof.
\end{proof}

Now let us start to prove Proposition \ref{exoticKE}.
Suppose that (b) holds. In light of Remark \ref{KESUR}, it suffices to check that map \eqref{imageWDG} is {\sl not} surjective.
  To this end, choose an  an  eigenbasis $B$    of  $V_E$ w.r.t.  $A_E$, and choose an order on $B$.
Using arguments of Subsection \ref{multEigen}, choose a $j$-element subset $\tilde{C} \subset B$ such that the function
$$e_{\tilde{C}} \colon \mathrm{spec}(A) \to \Z_{+}$$ coincides with $e$ and therefore
enjoys  properties (i)-(iv). Then
$$y_{\tilde{C}} \in \wedge^j(V^{*})(\lambda_1\lambda_2).$$

 I claim that $y_{\tilde{C}}$ does {\sl not} lie in the image of map  \eqref{imageWDG}.
Indeed, $\wedge_E^{j-2}(V_E^{*})(\lambda_1)$ is generated as the $E$-vector space by elements of the form $y_S$, where $S$ are  $(j-2)$-element subsets of $B$ such that
$\prod_{b\in S}\pi(b)=\lambda_1$.
On the other hand, $ \wedge_E^2(V_E^{*})(\lambda_2)$ is generated as the $E$-vector space by elements of the form $y_T$, where $T$ are  $2$-element subsets of $B$ such that
$\prod_{b\in T}\pi(b)=\lambda_2$. This implies that the image of map \eqref{imageWDG}  is generated as the $E$-vector space by all $y_B \wedge y_T$. If $S$ meets $T$ then 
%it follows from the
%very definition of $y_S$ and $y_T$  and basic properties of wedge products that 
(as we have already seen)
$y_S \wedge y_T=0$. If $S$ does {\sl not} meet $T$ then $S\cup T=S \sqcup T$ is a $j$-element subset of $B$ and $y_S \wedge y_T=\pm y_{S\cup T}$.

\begin{lem}
\label{CversusST}
 The $j$-element  $\tilde{C}$ does {\sl not} coincide with any of $S\cup T$. 
\end{lem}

 \begin{proof}[Proof of Lemma \ref{CversusST}] Suppose $\tilde{C}=S\cup T$. This implies that $\tilde{C}$ contains a subset $T$ that consists of two distinct elements say $x_1,x_2\subset B$
with $\pi(x_1)\pi(x_2)=\lambda_2$. So,  $\tilde{C}$ contains these $x_1,x_2$. It follows from the definition of $e_{\tilde{C}}$ \eqref{eCdef} that if we put
$$\alpha_1=\pi(x_1), \ \alpha_2=\pi(x_2)$$
then $\alpha_1,\alpha_2 \in \mathrm{spec}(A)$ and
$$\alpha_1 \alpha_2=\lambda_2, \ e(\alpha_1)\ge 1, e(\alpha_2)\ge 1, \ e(\alpha_1)+ e(\alpha_2) \ge 2.$$
If $\alpha_1 \ne \alpha_2=\lambda_2/\alpha_1$ then  $e(\alpha_1)\ge 1, e(\lambda_2/\alpha_1)\ge 1$, which violates property (iv).
If $\alpha_1=\alpha_2$ then $\alpha_2=\lambda_2/\alpha_1$. It follows that $\alpha_1=\pi(x_1)=\pi(x_2)$ and therefore $e(\alpha_1)\ge 2$, which also contradicts property (iv).  This ends the proof.
\end{proof}

{\sl End of Proof of Proposition \ref{exoticKE}}. %(modulo Lemma \ref{CversusST}).
Taking into account that the set of all $y_C$'s where $C$ runs through all $j$-element subsets of $B$
is {\sl linearly independent}, we conclude
 that $y_{\tilde{C}}$ cannot be presented as a $E$-linear combination of $ y_{S\cup T}$'s and therefore does {\sl not} lie in the image of  map \eqref{imageWDG}. Hence, (a) holds. We proved that (b) implies (a).

Conversely, suppose that (a) holds, i.e., the map \eqref{imageWDG} is {\sl  not} surjective.
Then there is a {\sl basis vector} $y_C \in  \wedge^j(V^{*})(\lambda_1\lambda_2)$ that does {\sl not} belong to the image where $C$ is a certain $j$-element subset of $B$ such that 
$$\prod_{x\in C}\pi(x)=\lambda_1\lambda_2.$$
I claim that the function $e:=e_C$ enjoys the properties (b). By definition, $e_C$ enjoys the propetries (bi), (bii), (biii). In particular,
$$\prod_{\alpha\in C}\alpha^{e_C(\alpha)}=\prod_{x \in C}\pi(x)=\lambda_1 \lambda_2.$$

Suppose that $e_C$ does {\sl not} enjoy  the property (biv).  This means that there exists $\alpha\in \mathrm{spec}(A)$ such that
$e_C(\alpha) \ne 0$ (i.e., $\alpha \in \pi(C)$) such that
$$\frac{\lambda_2}{\alpha} \in \mathrm{spec}(A), \ e_C\left(\frac{\lambda_2}{\alpha}\right)\ne 0$$
(i.e., 
$$\lambda_2/\alpha\in \pi(C) \subset  \mathrm{spec}(A).)$$Iin addition, if $\lambda_2/\alpha=\alpha$ then $e_C(\alpha)\ge 2$ (i.e.,
$\pi^{-1}(\alpha)$ contains at least two elements of $C$).

This implies that there are {\sl distinct} elements $x_1$ and $x_2$ of $C$ such that
$$\pi(x_1)=\alpha, \ \pi(x_2)=\frac{\lambda_2}{\alpha}.$$
Then $C$ coincides with the disjoint union of the $2$-element subset $T=\{x_1,x_2\}$ and the $(j-2)$-element subset $S=C \setminus T$. By definition of $T$,
$$\prod_{\alpha\in T}\alpha^{e_T(\alpha)}=\prod_{x \in T}\pi(x)=\pi(x_1) \cdot \pi(x_2)=\alpha \cdot \frac{\lambda_2}{\alpha}=\lambda_2.$$
Hence,
$$\prod_{\alpha\in T}\alpha^{e_T(\alpha)}=\lambda_2.$$
Since $C$ is the disjoint union of $S$ and $T$, we get
$$\prod_{\alpha\in S}\alpha^{e_S(\alpha)}=\frac{\prod_{\alpha\in C}\alpha^{e_C(\alpha)}}{\prod_{\alpha\in T}\alpha^{e_T(\alpha)}}=\frac{\lambda_1\lambda_2}{\lambda_2}=\lambda_1.$$
%This implies that
Hence,
$$\prod_{\alpha\in S}\alpha^{e_S(\alpha)}=\lambda_1.$$
It follows from Lemma \ref{imageWedgePr} that $y_C$ lies in the image of the map \eqref{imageWDG}, which is not true. The obtained contradiction implies that $e_C$ enjoys the property (biv).

This ends the proof of of Proposition \ref{exoticKE}.
\end{proof}

The following assertion will be used in the proof of Theorem \ref{semigroupTate} (with $K=\Q_{\ell}, V=V_{\ell}(X^n), A=\Fr_{X^n}$).

\begin{prop}
\label{lastEff}
We keep the notation and assumption of Subsection \ref{multEigen}.
Assume additionally that $\fchar(K)=0$,  $\dim_K(V)$ is even, and $A:V \to V$ is invertible. Let $q$ be a nonzero element of $K$ that is not a root of unity.
Let $h$ and $m$ be  positive integers that  enjoy the following properties.

\begin{itemize}
\item[(i)]  $h<m \le \dim_K(V)/2$.

%Let $m>h$ be an integer that is strictly less than $\dim_K(V)/2$ and enjoys the following properties.
\item[(ii)] 
If $e \colon \mathrm{spec}(A) \to \Z_{+}$ is any nonnegative integer-valued function such that
$$\sum_{\alpha\in  \mathrm{spec}(A)} e(\alpha)=2m, \ \prod_{\alpha\in  \mathrm{spec}(A)}\alpha^{e(\alpha)}=q^m$$
then there exist positive integers $j_1,j_2$ and  nonnegative integer-valued functions
$$f_1 \colon \mathrm{spec}(A) \to \Z_{+},  f_2 \colon \mathrm{spec}(A) \to \Z_{+}$$
such that
$$m=j_1+j_2,   \  j_2 \le h;$$
$$   e(\alpha)=f_1(\alpha)+f_2(\alpha)\ \forall \alpha \in \mathrm{spec}(A);$$
$$\sum_{\alpha\in  \mathrm{spec}(A)} f_1(\alpha)=2j_1, \ \prod_{\alpha\in  \mathrm{spec}(A)}\alpha^{f_1(\alpha)}=q^{j_1},$$
$$\sum_{\alpha\in  \mathrm{spec}(A)} f_2(\alpha)=2j_2, \ \prod_{\alpha\in  \mathrm{spec}(A)}\alpha^{f_2(\alpha)}=q^{j_2}.$$
\end{itemize}

Then 
\begin{equation}
\label{imageWdgePowerQ}
\sum_{j=1}^h \Lambda_{2(m-j),2j,K}\left(\wedge_K^{2(m-j)}(V^{*})(q^{m-j}) \otimes_K \wedge_K^{2j}(V^{*})(q^j)\right)=\wedge_K^{2m}(V^{*})(q^m).
\end{equation}
\end{prop}

\begin{proof}
In light of Remark \ref{restrictionLambdaK}  and arguments of Subsection \ref{wedgeImage}, it suffices to check that

\begin{equation}
\label{imageWdgePowerQE}
\sum_{j=1}^h \Lambda_{2(m-j),2j,E}\left(\wedge_E^{2(m-j)}(V_E^{*})(q^{m-j}) \otimes_E \wedge_E^{2j}(V_E^{*})(q^j)\right)=\wedge_E^{2m}(V_E^{*})(q^m).
\end{equation}
Recall that (in the notation of  Subsection \ref{multEigen}) that $B$ is an (ordered) eigenbasis of $V_E$ and
 all the $2m$-element subsets $C\subset \mathrm{spec}(A)$
with $\prod_{\alpha\in C}\alpha=q^m$ 
give rise to the base $\{y_C=\wedge_{x\in C}x^{*}\}$ of $\wedge_E^{2m}(V_E^{*})(q^m)$.
So, it suffices to prove that each such $y_C$ lies in one of the summands in LHS of \eqref{imageWdgePowerQE}.
To this end, let us consider the nonnegative integer-valued function
$$e_C \colon \mathrm{spec}(A)\to \Z_{+}, \ 
e(\alpha)=\#(C(\alpha)) \ \text{ where }\ C(\alpha):=\{x\in C \subset B\mid \pi(x)=\alpha\}.$$
(see \eqref{eCdef}). Clearly,
$$\sum_{\alpha\in \mathrm{spec}(A)}e_C(\alpha)=\#(C)=2m, \  \prod_{\alpha\in \mathrm{spec}(A)}\alpha^{e_C(\alpha)}=\prod_{\alpha\in C}\alpha =q^m.$$
By property (ii),  there exist positive integers $j_1,j_2$ and  nonnegative integer-valued functions
$$f_1 \colon \mathrm{spec}(A) \to \Z_{+},  f_2 \colon \mathrm{spec}(A) \to \Z_{+}$$
such that
$$m=j_1+j_2,   \  j_2 \le h;$$
$$   e(\alpha)=f_1(\alpha)+f_2(\alpha)\  \forall \alpha \in \mathrm{spec}(A);$$
$$\sum_{\alpha\in  \mathrm{spec}(A)} f_1(\alpha)=2j_1, \ \prod_{\alpha\in  \mathrm{spec}(A)}\alpha^{f_1(\alpha)}=q^{j_1},$$
$$\sum_{\alpha\in  \mathrm{spec}(A)} f_2(\alpha)=2j_2, \ \prod_{\alpha\in  \mathrm{spec}(A)}\alpha^{f_2(\alpha)}=q^{j_2}.$$
Let us partition each $C(\alpha)$ into a {\sl disjoint union} of two sets
$$C(\alpha)=C(\alpha)_1\cup C(\alpha)_2   \text{ with } \ C(\alpha)_1\cap C(\alpha)_2 =\emptyset, \ \#(C(\alpha)_1)=f_1(\alpha),  \#(C(\alpha)_2)=f_2(\alpha)$$
and define $C_1$ (resp. $C_2)$ as the (disjoint) union of all $C(\alpha_1)$ (resp. of all $C(\alpha_2)$). Then $C$ becomes a {\sl disjoint union} of $C_1$ and $C_2$, and
$$f_1=e_{C_1},  f_2=e_{C_2}.$$
It follows that
$$\sum_{\alpha\in  \mathrm{spec}(A)} e_{C_1}(\alpha)=2j_1, \ \prod_{\alpha\in  \mathrm{spec}(A)}\alpha^{e_{C_1}(\alpha)}=q^{j_1},$$
$$\sum_{\alpha\in  \mathrm{spec}(A)} e_{C_2}(\alpha)=2j_2, \ \prod_{\alpha\in  \mathrm{spec}(A)}\alpha^{e_{C_2}(\alpha)}=q^{j_2}.$$

This implies that  
$$y_{C_1} \in \wedge_E^{2(j_1)}(V_E^{*})(q^{j_1})=\wedge_E^{2(m-j_2)}(V_E^{*})(q^{m-j_2}),
y_{C_2} \in \wedge_E^{2(j_2)}(V_E^{*})(q^{j_2}).$$
Since $C$ is a disjoint union of $C_1$ and $C_2$,
$$y_C =\pm y_{C_1}\wedge y_{C_2}\in
\Lambda_{2(m-j_2),2j_2,E}\left(\wedge_E^{2(m-j_2)}(V_E^{*})(q^{m-j_2}) \otimes_E \wedge_E^{2j_2}(V_E^{*})(q^{j_2})\right).$$
In order to finish the proof, one has only to recall that $j \le h_2$.
\end{proof}

\section{Tate forms}
\label{TateProof}
\begin{sect}
Recall that $X$ is an  abelian variety of positive dimension $g$ over a finite field $k$ with $\fchar(k)=p$ and  $\#(k)=q$.
Let $\ell\ne p$ be a prime and $T_{\ell}(X)$ the $\ell$-adic Tate module of $X$.  Let us consider the corresponding $\Q_{\ell}$-vector space
$$V_{\ell}(X)=T_{\ell}(X)\otimes_{\Z_{\ell}}\Q_{\ell},$$
which is a $2g$-dimensional vector space over $\Q_{\ell}$. The action of $\Fr_X$ extends by $\Q_{\ell}$-linearity to $V_{\ell}(X)$. So, we may view
$\Fr_X$ as a $\Q_{\ell}$-linear automorphism of $V_{\ell}(X)$, whose characteristic polynomial coincides with $\mathcal{P}_X(t)$. A theorem of
Weil \cite{Mumford,Tate1} asserts that $\Fr_X$ acts as a {\sl semisimple} linear operator in $V_{\ell}(X)$. 
Let $\bar{\Q}_{\ell}$ be an algebraic closure of $\Q_{\ell}$.  Let us choose a field embedding 
$$L_X =\Q(R_X)\hookrightarrow \bar{\Q}_{\ell}.$$
Further we will identify $L_X$ with its image in $\bar{\Q}_{\ell}$. We have
$$R_X \subset L_X \subset \bar{\Q}_{\ell}.$$
Let us consider the $2\dim(X)$-dimensional $\bar{\Q}_{\ell}$-vector space
$$\bar{V}_{\ell}(X):=V_{\ell}(X)\otimes_{\Q_{\ell}}\bar{\Q}_{\ell}.$$
Extending the action of $\Fr_X$ by $\bar{\Q}_{\ell}$-linearity, we get a $\bar{\Q}_{\ell}$-linear operator
$$\overline{\Fr}_X \colon \bar{V}_{\ell}(X)\to \bar{V}_{\ell}(X), \ v\otimes \lambda\mapsto \Fr_X(v)\otimes \lambda \ \forall v \in V_{\ell}(X), \lambda \in \bar{\Q}_{\ell}.$$
In the notation of Section \ref{linAlgebra}, let us put
%$$\kappa: \Q(R_X)=L_X \hookrightarrow \bar{Q}_{\ell}.$$
\begin{equation}
\label{notationQl}
K=\Q_{\ell},  V=V_{\ell}(X),  A=\Fr_X \colon V_{\ell}(X)\to V_{\ell}(X),    E=\bar{\Q}_{\ell}.
\end{equation}
Then
\begin{equation}
\label{notationQlbar}
V_E=\bar{V}_{\ell}(X), A_E=\overline{\Fr}_X; \ \mathrm{spec}(A)=R_X,  \mult_A=\mult_X:R_X \to \Z_{+}.
\end{equation}
\end{sect}

\begin{rem}
\label{YXm} If $m$ is a positive integer and $Y=X^m$ then it is well known that there is a canonical isomorphism of $\Q_{\ell}$-vector spaces
$$V_{\ell}(Y)=\oplus_{i=1}^m V_{\ell}(X)$$ such that $\Fr_Y$ acts on $V_{\ell}(Y)$ as
$$\Fr_Y(x_1,\dots x_m)=(\Fr_X x_1, \dots, \Fr_X x_m) \ \forall (x_1,\dots x_m)\in \oplus_{i=1}^m V_{\ell}(X)=V_{\ell}(Y).$$
This implies that
\begin{equation}
\label{RXY}
\mathcal{P}_Y(t)=\mathcal{P}_X(t)^m, \ R_Y=R_X,  L_X=L_Y, \mult_Y(\alpha)=m \cdot \mult_X(\alpha) \ \forall \alpha\in R_X=R_Y.
\end{equation}
In particular,
\begin{equation}
\label{multY}
\mult_Y(\alpha) \ge m \ \forall \alpha \in R_Y=R_X.
\end{equation}

\end{rem}
Any invertible sheaf/divisor class $\mathcal{L}$ on $X$
gives rise to (defined up to multiplication by an element of $\Q_{\ell}^{*}$) a $\Q_{\ell}$-bilinear alternating {\sl Riemann form}  (the first $\ell$-adic Chern class of  $\mathcal{L}$) \cite{Mumford}
$$\phi=\phi_{\mathcal{L}} \colon V_{\ell}(X) \times V_{\ell}(X) \to \Q_{\ell} $$
 such that 
 \begin{equation}
\label{Riemann}
\phi(\Fr_X (x),\Fr_X(y))=q \cdot \phi(x,y) \ \forall x,y \in  V_{\ell}(X).
\end{equation}

\begin{sect}
A theorem of Tate \cite{Tate1} asserts that every alternating $\Q_{\ell}$-bilinear form $\phi$ on $V_{\ell}(X)$ that satisfies \eqref{Riemann} is a $\Q_{\ell}$-linear combination
of forms of type $\phi_{\mathcal{L}}$. We call such a form an $\ell$-adic  Tate form of degree $2$
 and denote by $\tate_2(X,\ell)$ the subspace of all such forms in
$\Hom_{\Q_{\ell}}(\Lambda^2 V_{\ell}(X),\Q_{\ell})$. In other words,
$$\tate_2(X,\ell):=\{\phi \in \Hom_{\Q_{\ell}}(\Lambda^2 V_{\ell}(X),\Q_{\ell})\mid 
\phi(\Fr_X (x),\Fr_X(y))=q \cdot \phi(x,y)$$
$$  \forall x,y \in  V_{\ell}(X)\}.$$
More generally, let us define for each nonnegative integer $d\le \dim(X)=g$ the subspace $\tate_{2d}(X,\ell)$ of all alternating $2d$-forms
$\psi \in \Hom_{\Q_{\ell}}(\Lambda^{2d} V_{\ell}(X),\Q_{\ell})$
such that
$$\psi(\Fr_X (v_1), \dots, \Fr_X(v_{2d}))=q^d \cdot \psi(v_1,\dots, v_{2d})
 \ \forall x_1, \dots, x_{2d} \in  V_{\ell}(X).$$
We call elements of $\tate_{2d}(X,\ell)$ {\sl Tate forms of degree $2d$}\begin{footnote}{In \cite{ZarhinEssen} we called them admissible forms.}\end{footnote}.
\end{sect}

\begin{rem}
\label{tateBis}
Clearly, $\tate_{2d}(X,\ell)$ consists of all $\psi \in \Hom_{\Q_{\ell}}(\Lambda^{2d} V_{\ell}(X),\Q_{\ell})$
such that
%\begin{equation}
%\label{tateVar}
$$\psi(\Fr_X^{-1} (v_1), \dots, \Fr_X^{-1}(v_{2d}))=q^{-d} \cdot \psi(v_1,\dots, v_{2d}) \ \forall v_1, \dots, v_{2d} \in  V_{\ell}(X).$$
%\tate_{2d}(X,l)=\{\psi \in \Hom_{\Q_{\ell}}(\Lambda^{2d} V_{\ell}(X),\Q_{\ell})\mid 
%\end{equation}
%$$\psi(\Fr_X^{-1} (v_1), \dots, \Fr_X^{-1}(v_{2d}))=q^{-d} \cdot \psi(v_1,\dots, v_{2d}) \ \forall v_1, \dots, v_{2d} \in  V_{\ell}(X)\}=$$
Since $\Fr_X$ acts on $V_{\ell}(X)$ as $\rho_{\ell}(\sigma_k)$, the subspace $\tate_{2d}(X,l)$ consists of all $\psi \in \Hom_{\Q_{\ell}}(\Lambda^{2d} V_{\ell}(X),\Q_{\ell})$
such that
%$$\{\psi \in \Hom_{\Q_{\ell}}(\Lambda^{2d} V_{\ell}(X),\Q_{\ell})\mid $$
$$\psi(\rho_{\ell}(\sigma_k)^{-1} (v_1), \dots, \rho_{\ell}(\sigma_k)^{-1}(v_{2d}))=\chi_{\ell}(\sigma_k)^{-d} \cdot \psi(v_1,\dots, v_{2d}) \ \forall v_1, \dots, v_{2d} \in  V_{\ell}(X).$$
Since $\sigma_k$ is a topological generator of $\Gal(k)$, the subspace $\tate_{2d}(X,l)$ consists of all $\psi \in \Hom_{\Q_{\ell}}(\Lambda^{2d} V_{\ell}(X),\Q_{\ell})$
such that
%\begin{equation}
%\label{tateVar2}
$$\psi(\rho_{\ell}(\sigma)^{-1} (v_1), \dots, \rho_{\ell}(\sigma)^{-1}(v_{2d}))=\chi_{\ell}(\sigma)^{-d} \cdot \psi(v_1,\dots, v_{2d})$$
$$  \forall \sigma \in \Gal(k), v_1, \dots, v_{2d} \in  V_{\ell}(X).$$
%\tate_{2d}(X,\ell)=\{\psi \in \Hom_{\Q_{\ell}}(\Lambda^{2d} V_{\ell}(X),\Q_{\ell})\mid 
%\end{equation}
%$$\psi(\rho_{\ell}(\sigma)^{-1} (v_1), \dots, \rho_{\ell}(\sigma)^{-1}(v_{2d}))=\chi_{\ell}(\sigma)^{-d} \cdot \psi(v_1,\dots, v_{2d}) \ \forall \sigma \in \Gal(k); v_1, \dots, v_{2d} \in  V_{\ell}(X)\}.$$
\end{rem}

\begin{rem}
\label{TateEigenQ}
In the notation of  Section \ref{linAlgebra}, \eqref{notationQl} and \eqref{notationQlbar},
$$\tate_{2d}(X,\ell)=\wedge_K^{2d}(V^{*})(q^d)$$
is the eigenspace of  $K$-linear operator
$\wedge^{2d}(A^{*}) \colon \wedge_K^{2d}(V^{*}) \to \wedge_K^{2d}(V^{*}) $ attached to the eigenvalue $q^d$.
\end{rem}

\begin{sect}
For each integer $d\ge 3$ the exterior product map
$$\Hom_{\Q_{\ell}}(\Lambda^{2(d-1)} V_{\ell}(X),\Q_{\ell}) \otimes_{\Q_{\ell}} \Hom_{\Q_{\ell}}(\Lambda^{2} V_{\ell}(X),\Q_{\ell})
\to   \Hom_{\Q_{\ell}}(\Lambda^{2d} V_{\ell}(X),\Q_{\ell}) ,$$
$$   \phi\otimes \psi \mapsto \phi \wedge \psi$$
induces the $\Q_{\ell}$-linear map
\begin{equation}
\label{wedgeTate}
\tate_{2(d-1)}(X,\ell) \otimes_{\Q_{\ell}} \tate_{2}(X,\ell) \to \tate_{2d}(X,\ell), \ \phi\otimes\psi \mapsto \phi \wedge \psi.
\end{equation}

\begin{defn}
Let $d>1$ be an integer. An $\ell$-adic  {\sl Tate form} of degree $2d$ is called {\sl exceptional} if it does {\sl not} lie
in the 
%$\Q_{\ell}$-linear span of the
 image of map \eqref{wedgeTate}.
\end{defn}
\end{sect}

\begin{lem}
\label{exceptreduced}
Let $d$ be a positive integer such that
$$2 \le d \le \dim(X).$$
Let $\ell \ne p$ be a prime.
Then  the following conditions are equivalent.
\begin{itemize}
\item[(a)]
There exists an exceptional $\ell$-adic Tate form on $X$ of degree $2d$.
\item[(b)]
There exists an admissible reduced function $e \colon R_X \to \Z$ of weight $2d$ such that
$$0 \le e(\alpha) \le \mult_X(\alpha) \ \forall \alpha\in R_X.$$
\end{itemize}
\end{lem}

\begin{proof}
In the notation of \eqref{imageWDGK}, \eqref{notationQl} and \eqref{notationQlbar},  it follows from Remark \ref{TateEigenQ} that
property (b) is equivalent to the {\sl non-surjectiveness} of 
$$\wedge_K^{j-2}(V^{*})(\lambda_1) \otimes_K \wedge_K^2(V^{*})(\lambda_2) \to \wedge_K^{j}(V^{*})(\lambda_1\lambda_2),  \psi\otimes \phi \mapsto \psi\wedge \phi$$
with 
$$j=2d,  \lambda_1=q^{d-1}, \lambda_2=q, \lambda_1 \lambda_2=q^d.$$
By Proposition \ref{exoticKE}, the  non-surjectiveness  of this map is equivalent to the existence of a function
$e \colon \mathrm{spec}(A) \to \Z_{+}$ that enjoys properties (i)-(iv) of Proposition \ref{exoticKE}. Since $\mathrm{spec}(A)=R_X$, we may view $e$ as a function
$e \colon R_X \to \Z_{+}$.  Now property (ii) means that $e$ has weight $2d$, property (iii) that $e$ is admisible, and property (iv) that $e$ is reduced. As for property (i), it means that
$$e(\alpha) \le \mult_A(\alpha)=\mult_X(\alpha) \ \forall \alpha \in \mathrm{spec}(A)=R_X.$$
This implies that properties (i)-(iv) of Theorem \ref{exoticKE} are equivalent to property (b) of Lemma \ref{exceptreduced}.
It follows that
 properties (a) and (b) of  Lemma \ref{exceptreduced} are equivalent.
\end{proof}

 \begin{sect}
 \label{etaleT}{\bf Twists and Tate classes}.
 %Now let us discuss Tate classes.
%In order to state our second main result,  we need to 
Let us %fix an algebraic closure $\bar{k}$ of $k$ and 
consider an abelian variety $\bar{X}=X \times_k \bar{k}$ over the algebraic closure $\bar{k}$ of $k$
and its \'etale $\ell$-adic cohomology groups $\mathrm{H}^j(\bar{X},\Q_{\ell})$ \cite{Tate0,Kleiman,Katz,Berkovich}.  Here $\ell$ is any prime different from $\fchar(k)$,  $j$ any nonnegative integer, 
and $\mathrm{H}^j(\bar{X},\Q_{\ell})$ is a certain finite-dimensional $\Q_{\ell}$-vector space endowed with a continuous linear action of the absolute Galois group $\Gal(k):=\Gal(\bar{k}/k)$ of $k$.
We write
$$\chi_{\ell} \colon \Gal(k) \to \Z_{\ell}^{*}\subset \Q_{\ell}^{*}$$
for the $\ell$-adic {\sl cyclotomic character} (see Section \ref{neat} above). There exists a certain ``naturally defined'' one-dimensional
$\Q_{\ell}$-vector space $\Q_{\ell}(1)$  (that was denoted by $W$ in \cite[Sect. 2]{Tate0}) endowed with the natural continuous linear action of $\Gal(k)$ defined by the cyclotomic character
$$\chi_{\ell}  \colon \Gal(k) \to \Q_{\ell}^{*}=\Aut_{\Q_{\ell}}(\Q_{\ell}(1))$$
(see \cite{Tate0,Kleiman,Katz,Milne}). Namely,  $\Q_{\ell}(1)=\Z_{\ell}(1)\otimes_{\Z_{\ell}}\Q_{\ell}$ where  $\Z_{\ell}(1)$ is the projective limit of multiplicative groups (finite Galois modules) $\mu_{\ell^n}$
of $\ell^n$th roots of unity in $\bar{k}$.

Let us fix once and for all an {\sl $\ell$-adic orientation}, i.e., an isomorphism of $\Q_{\ell}$-vector spaces
$$\Q_{\ell}(1) \cong \Q_{\ell},$$
which allows us to identify $\Q_{\ell}$ not only with $\Q_{\ell}(1)$ but also with all {\sl tensor powers} $\Q_{\ell}(i)$ \cite{Tate0,Kleiman,Katz,Berkovich,Tate94} of $\Q_{\ell}(1)$.

 Let $i$ be an integer. Let us consider the {\sl twist} $\mathrm{H}^j(\bar{X},\Q_{\ell})(i)$
of the Galois module $\mathrm{H}^j(\bar{X},\Q_{\ell})$
by character $\chi_{\ell}^i$ of $\Gal(k)$ \cite{Tate0,Kleiman,Katz,Tate94}. In other words, $\mathrm{H}^j(\bar{X},\Q_{\ell})(i)$ coincides with $\mathrm{H}^j(\bar{X},\Q_{\ell})$ as the $\Q_{\ell}$-vector space but if 
$$ \sigma,c \mapsto \sigma(c) \ \forall \sigma \in \Gal(k), \ c \in \mathrm{H}^j(\bar{X},\Q_{\ell})$$
is the Galois action on $\mathrm{H}^j(\bar{X},\Q_{\ell})$
then in $\mathrm{H}^j(\bar{X},\Q_{\ell})(i)$ a Galois automorphism $\sigma$ sends $c$ to $\chi_{\ell}(\sigma)^i\sigma(c)$.

If $W$ is a Galois-invariant $\Q_{\ell}$-vector subspace in $\mathrm{H}^j(\bar{X},\Q_{\ell})$  then we write $W(i)$ for the same $\Q_{\ell}$-vector subspace in $\mathrm{H}^j(\bar{X},\Q_{\ell})(i)$.
Clearly, $W(i)$ is a Galois-invariant subspace of $\mathrm{H}^j(\bar{X},\Q_{\ell})(i)$ but {\sl not} necessarily isomorphic to $W$ as Galois module.

\begin{rems}
\label{twistCup}
\begin{itemize}
\item[(i)]
If $j_1,j_2$ are any nonnegative integers then the Galois-equivariant $\Q_{\ell}$-bilinear cup product in the cohomology of $\bar{X}$ leads to a Galois-equivariant $\Q_{\ell}$-linear map
\begin{equation}
\label{cupH}
\mathrm{H}^{j_1}(\bar{X},\Q_{\ell})\otimes_{\Q_{\ell}} \mathrm{H}^{j_2}(\bar{X},\Q_{\ell}) \to \mathrm{H}^{j_1+j_2}(\bar{X},\Q_{\ell}), \ c_1\otimes c_2 \mapsto c_1 \cup c_2,
\end{equation}
which, in turn, gives rise to 
 the natural  Galois-equivariant $\Q_{\ell}$-linear map \cite{Tate0,Kleiman}
\begin{equation}
\label{cupProduct}
\mathrm{H}^{j_1}(\bar{X},\Q_{\ell})(i_1)\otimes_{\Q_{\ell}} \mathrm{H}^{j_2}(\bar{X},\Q_{\ell})(i_2) \to \mathrm{H}^{j_1+j_2}(\bar{X},\Q_{\ell})(i_1+i_2),  \ c_1\otimes c_2 \mapsto c_1 \cup c_2.
\end{equation}
\item[(ii)]
Let $W_1$ (resp. $W_2$) is s a Galois-invariant $\Q_{\ell}$-vector subspace in $\mathrm{H}^{j_1}(\bar{X},\Q_{\ell})$  (resp. in $\mathrm{H}^{j_2}(\bar{X},\Q_{\ell})$) 
and $W \subset \mathrm{H}^{j_1+j_2}(\bar{X},\Q_{\ell})$ be the image of subspace
$$W_1\otimes_{\Q_{\ell}}W_2 \subset \mathrm{H}^{j_1}(\bar{X},\Q_{\ell})\otimes_{\Q_{\ell}} \mathrm{H}^{j_2}(\bar{X},\Q_{\ell})$$
under the map \eqref{cupH}. It follows readily that the twist
$$W(i_1+i_2)\subset \mathrm{H}^{j_1+j_2}(\bar{X},\Q_{\ell})(i_1+i_2)$$
coincides with the image of subspace
$$W_1(i_1)\otimes_{\Q_{\ell}}W_2(i_2) \subset \mathrm{H}^{j_1}(\bar{X},\Q_{\ell})(i_1)\otimes_{\Q_{\ell}} \mathrm{H}^{j_2}(\bar{X},\Q_{\ell})(i_2)$$
under the map \eqref{cupProduct}.
\end{itemize}
\end{rems}

\begin{defn}
\label{WlT}
Let $d$ be a nonnegative integer. Let us consider the $\Q_{\ell}$-vector subspace 
$$\T_{\ell,d}(X):=\mathrm{H}^{2d}(\bar{X},\Q_{\ell})(d)^{\Gal(k)}$$ of {\sl Galois invariants} in $\mathrm{H}^{2d}(\bar{X},\Q_{\ell})(d)$
and a {\sl weight} $\Q_{\ell}$-vector  subspace
$$W_{\ell,d}(X):=\{c \in \mathrm{H}^{2d}(\bar{X},\Q_{\ell})\mid \sigma(x)=\chi_{\ell}(\sigma)^{-d}c \ \forall \sigma \in \Gal(k)\}$$
in $\mathrm{H}^{2d}(\bar{X},\Q_{\ell})$. It follows from the very definitions that
\begin{equation}
\label{TateTwistClass}
\T_{\ell,d}(X)=W_{\ell,d}(X)(d).
\end{equation}
\end{defn}

\begin{rem}
\label{product TateCl}
Let $d_1$ and $d_2$ be nonnegative integers.  It follows from the Galois equivariance  of maps \eqref{cupH} and \eqref{cupProduct} combined with Remark \ref{twistCup}(ii)
that  the image $W_{\ell,d_1,d_2}(X)$ of 
$$W_{\ell,d_1}(X)\otimes W_{\ell,d_2}(X)\to \mathrm{H}^{2(d_1+d_2)}(\bar{X},\Q_{\ell}), \   c_1\otimes c_2 \to c_1\cup c_2$$
lies in $W_{\ell,d_1+d_2}(X)$.
Similarly, the image $\T_{\ell,d_1,d_2}(X)$ of 
$$\T_{\ell,d_1}(X)\otimes \T_{\ell,d_2}(X)\to \mathrm{H}^{2(d_1+d_2)}(\bar{X},\Q_{\ell})(d_1+d_2), \   c_1\otimes c_2 \to c_1\cup c_2$$
lies in $\T_{\ell,d_1+d_2}(X)$. In addition, it follows from \eqref{TateTwistClass} that
$$\T_{\ell,d_1,d_2}(X)=W_{\ell,d_1,d_2}(X)(d_1+d_2).$$
\end{rem}

\begin{defn}
\label{tateDeinition}
Let $\ell \ne \fchar(k)$ be a prime and  $d$  a nonnegative integer.  
%Let $c \in \mathrm{H}^{2d}(\bar{X},\Q_{\ell})(d)$ be a (twisted) $\ell$-adic cohomology class.
\begin{itemize}
\item[(i)]
Elements of $\T_{\ell,d}(X)$ are called $\ell$-adic  {\sl Tate classes of dimension $2d$} on $X$.
%Class $c$ is called an $\ell$-adic  {\sl Tate class} of dimension $2d$ on $X$ if it is a {\sl Galois-invariant} element of $\mathrm{H}^{2d}(\bar{X},\Q_{\ell})(d)$.
\item[(ii)]
A nonzero $2d$-dimensional Tate class $c$ is called {\sl exotic} if $d\ge 1$ and $c$ cannot be presented as a linear combination of products of $d$ Tate classes of dimension 2 with coefficients in $\Q_{\ell}$.
\item[(iii)]
A Tate class $c$ of dimension $2d$ is called {\sl very exotic} if $d\ge 1$ and $c$ cannot be presented as a linear combination  with coefficients in $\Q_{\ell}$ of products of  Tate classes of dimension $2d-2$ and $2$, i.e.,
$c$ does {\sl not} belong to $\T_{\ell,d-1,1}(X)$.
\end{itemize}
\end{defn}

\begin{rems}
\label{remTate}
\begin{itemize}
\item[]
\item[(i)]
Let $Z$ be a closed irreducible subvariety  of codimension $d$ in $X$. The choice of the $\ell$-adic orientation allows us to define
 the $\ell$-adic class $\mathrm{cl}(Z)\in \T_{\ell,d}(X)\subset\mathrm{H}^{2d}(\bar{X},\Q_{\ell})(d)$ of $Z$ \cite{Tate0,Tate94}.
%A theorem of Tate \cite{Tate1} asserts that $\T_{\ell,2}(X)$ is spanned by {\sl divisor classes}.
%each $2$-dimensional Tate class is a linear combination of 
%with coefficients in $\Q_{\ell}$.
Tate  \cite{Tate0,Tate94} conjectured that for all nonnegative integer $d$ the subspace  $\T_{\ell,d}(X)$ is spanned by all  $\mathrm{cl}(Z)$ 
and proved it for $d=1$ \cite{Tate1}.
%every $2d$-dimensional Tate class is a linear combination of 
%$\ell$-adic {\sl classes} of codimension $d$ closed irreducible subvarietis in $X$.
% with coefficients in $\Q_{\ell}$ 

\item[(ii)]
If $d$ is a nonnegative integer and $d \le g$ then it is known  \cite{Tate0,Tate94} that $\T_{\ell,d}(X) \ne \{0\}$.

%Clearly, all non-exotic (resp. not very exotic) $2d$-dimensional $\ell$-adic Tate classses constitute a $\Q_{\ell}$-vector subspace in $\T_{\ell,d}(X)$.
%\item[(iii)]
%The Galois-equivariance of the cup product \eqref{cupProduct} implies that a cup product of $\ell$-adic Tate classes of dimensions $2d_1$ and $2d_2$
%is an $\ell$-adic Tate class of dimension $2(d_1+d_2)$.
\item[(iii)]
Clearly, all $2$-dimensional Tate classes are neither exotic nor  very exotic. Hence,
the existence of an exotic (or very exotic) Tate class of dimension $d$ implies readily that
$$2\le d \le g=\dim(X),$$
because $\mathrm{H}^{2d}(\bar{X},\Q_{\ell})=\{0\}$ for all $d>g$,
and, therefore, all $2d$-dimensional Tate classes are just zero. (Actually, it is known that
$\mathrm{H}^{2g}(\bar{X},\Q_{\ell})(g)$ is a one-dimensional $\Q_{\ell}$-vector space generated by the $g$th self-product of the class of a hyperplane
section of $X$ \cite{Tate0}. Therefore, there are no non-exotic Tate classes of dimension $2g$.)

\item[(iv)]
Clearly, every {\sl very exotic} Tate class is {\sl exotic}. 

Conversely, suppose that  there exists an {\sl exotic} $2d$-dimensional $\ell$-adic Tate class on $X$. I claim that there is a positive integer $d^{\prime} \le d$ such that
there exists a {\sl very exotic} $2d^{\prime}$-dimensional $\ell$-adic Tate class on $X$. Indeed, decreasing $d$ if necessary, we may and will assume that $d$ 
 is the smallest positive integer such that there is an exotic $\ell$-adic $2d$-dimensional Tate class  on $X$. Let $c$ be such a class. Then  $d>1$.
 
Assume that $c$ is {\sl not} very exotic. Then $c$ is a linear combination of cup products $h_i \cup c_i$ where
all $c_i$ are nonzero Tate classes of dimension $2$ and all $h_i$ are nonzero Tate classes of dimension $2(d-1)$. Since $c$ is exotic, 
 there is an index $i$ such that $h_i$ is exotic. But exotic $h_i$ is $2(d-1)$-dimensional, which contradicts the minimality of $d$. The obtained contradiction
 implies that $c$ itself is very exotic, which ends the proof.

It follows that {\sl $X$ carries an $\ell$-adic exotic Tate class  if and only if it carries a very exotic $\ell$-adic Tate class} (may be, of different dimension).

\end{itemize}
\end{rems}

%recall that an $\ell$-adic Tate cohomology class (see Section \ref{TateProof} below) 

\end{sect}

\begin{sect}
\label{etale}
{\bf \'Etale cohomology of abelian varieties}.
Let us consider the abelian variety $\bar{X}=X\times_k\bar{k}$ over $\bar{k}$.
Let $j$ be a nonnegative integer and  let $\mathrm{H}^{j}(\bar{X},\Q_{\ell})$ be the $j$th \'etale $\ell$-adic cohomology group of $\bar{X}$, which is a finite-dimensional $\Q_{\ell}$-vector space endowed with the canonical continuous linear action of $\Gal(k)$ \cite{Tate0,Kleiman,Milne}.
There is a canonical $\Gal(k)$-equivariant isomorphism of graded $\Q_{\ell}$-algebras  (\cite{Tate0,Kleiman,Berkovich}, \cite[Sect. 12]{Milne})
\begin{equation}
\label{HopfIso}
\oplus_{j=0}^{2\dim(X)} \mathrm{H}^{j}(\bar{X},\Q_{\ell})\cong \oplus_{j = 0}^{2\dim(X)} \Hom_{\Q_{\ell}}(\Lambda^{j}_{\Q_{\ell}} V_{\ell}(X),\Q_{\ell}).
\end{equation}
Its Galois equivariance combined with Remark \ref{tateBis} imply that (in the notation of Definition \ref{WlT}) the map \eqref{HopfIso}  induces for all nonnegative integers $d \le \dim(X)$ a $\Q_{\ell}$-linear isomorphism between
$$W_{\ell,d,}(X)=\{c \in \mathrm{H}^{2d}(\bar{X},\Q_{\ell})\mid \sigma(c)=\chi_{\ell}(\sigma)^{-d}c \ \forall \sigma \in \Gal(k) \}=$$
$$\{c \in \mathrm{H}^{2d}(\bar{X},\Q_{\ell})\mid \sigma_k(c)=\chi_{\ell}(\sigma_k)^{-d} c\}
\subset \mathrm{H}^{2d}(\bar{X},\Q_{\ell})$$
and the subspace 
              $$\tate_{2d}(X,\ell)\subset \Hom_{\Q_{\ell}}(\Lambda^{2d}_{\Q_{\ell}} V_{\ell}(X),\Q_{\ell})$$
              of $\ell$-adic Tate forms of degree $2d$ on $X$.
              Recall (see Definitions \ref{WlT} and \ref{tateDeinition}) that 
              %Clearly, in the notation of Subsection \ref{etaleT}, 
              the twist 
              $W_{\ell,d}(X)(d)\subset  \mathrm{H}^{2d}(\bar{X},\Q_{\ell})(d)$ coincides with the subspace  $\T_{\ell,d}(X)$  
              of $2d$-dimensional $\ell$-adic Tate classes on $X$.
              %Galois-invariants in $\mathrm{H}^{2d}(\bar{X},\Q_{\ell})(d)$ ,
              %i.e., with the subspace of all
              %$2d$-dimensional  $\el\ell$-adic Tate classes  on $X$.
Recall that map \eqref{HopfIso} is  a $\Q_{\ell}$-algebra isomorphism. 
%Combining Remarks \ref{twistCup} and 
Applying Remark \ref{product TateCl}  to $d_1=d-1$ and $d_2=1$, we obtain that {\sl the existence of a very exotic  $\ell$-adic Tate class of dimension $2d$ on $X$ is equivalent to the existence of an exceptional $\ell$-adic Tate form of degree $2d$ on $X$.}

\end{sect}

We will need to state explicitly the following useful assertion.

\begin{lem}
\label{equivTateForms}
Let $X$ be an abelian variety over $k$. Let $\ell \ne \fchar(k)$ be a prime.
Then the  following three conditions are equivalent.

\begin{itemize}
\item[(i)]
$X$ carries an exotic $\ell$-adic Tate class.

\item[(ii)]
$X$ carries a very exotic $\ell$-adic Tate class.
\item[(iii)]
There exists an exceptional $\ell$-adic Tate form  on $X$.
\end{itemize}

In addition, the validity of equivalent conditions (i)-(iii) does not depend on a choice of $l$.
\end{lem}

\begin{proof}[Proof of Lemma \ref{equivTateForms}]
The equivalence of (ii) and (iii) follows readily from the arguments at the end of Subsection \ref{etale}.
The equivalence of (i) and (ii) was already proven in Remark \ref{remTate}(iv).

Notice that property (b) of Lemma  \ref{exceptreduced} does not depend on the choice of $\ell$. Now Lemma \ref{exceptreduced} implies that 
the validity of (iii)  does {\sl not} depend on the choice of $\ell$.
\end{proof}

\begin{proof}[Proof of Theorem \ref{mainTate}]

By a theorem  of Tate \cite{Tate1}, if $Y$ is any abelian variety over $k$ (e.g., $Y=X^n$) then every element of $\T_{\ell,1}(Y)$ 
is a linear combination of divisor classes on $Y$ with coefficients in $\Q_{\ell}$.

%Let $Y$ be any abelian variety over $k$ (e.g., $Y=X^n$). It follows from Remark \ref{remTate}(iv)
%that  $Y$ does {\sl not}  carry carries an exotic $\ell$-adic Tate class if and only if it does not carry aa very exotic $\ell$-adic Tate class as well. Arguments of Subsection \ref{etale} applied to $Y$ instead of $X$ imply that
%there is an exceptional $\ell$-adic Tate form of some degree $2d$ on $Y$.
%on $X^n$ for some positive integer $n$.

Suppose that there is an  exotic $\ell$-adic Tate class on $X^n$ for some positive integer $n$.  It follows from  Lemma \ref{equivTateForms} (applied to $Y=X^n$
instead of $X$) that there is an exceptional $\ell$-adic Tate form on $X^n$. 
%for some positive integer $n$.
In light of Lemma \ref{exceptreduced} (applied to $Y$ instead of $X$),
%combined with arguments of Subsection \ref{etale} applied to $Y=X^n$,
 there exists an admissible reduced function
$R_X=R_Y \to \Z_{+}$. In light of Theorem \ref{mainRelation}, there exists an  admissible reduced function
$e \colon R_X \to \Z_{+}$ of weight $\le N(g)$. This means that
$$\wt(e)=\sum_{\alpha\in R_X}e(\alpha) \le 2 N(g);$$
in particular,
$$0 \le e(\alpha) \le 2 N(g) \ \forall \alpha \in R_X.$$
Let us put $Z=X^{2N(g)}$ and consider $e$ as the reduced admissible function
$$R_Z=R_X \to \Z_{+},  \ \alpha \mapsto e(\alpha).$$
In light of Remark \ref{YXm} applied to $m=2N(g)$,
$$e(\alpha) \le  2 N(g)\le \mult_Z(\alpha) \ \forall \alpha \in R_Z=R_X.$$
It follows from Lemma \ref{exceptreduced} that  there is an exceptional $\ell$-adic Tate form on $X^{2N(g)}=Z$.
Applying Lemma  \ref{equivTateForms} to $Z$, we obtain that  there is an  exotic $\ell$-adic Tate class on $Z=X^{2N(g)}$.
%In order to finish the proof, one has only to recall that according to a theorem of Tate \cite{Tate1},
%the subspace of Tate classes in  $\mathrm{H}^{2}(\bar{Z},\Q_{\ell})$
%consists of all linear combinations of all  divisor classes on $Z$ and apply  arguments of Subsection \ref{etale} to $Z$ (instead of $X$).
Now the last assertion of Lemma \ref{equivTateForms} implies that  there is an  exotic $l$-adic Tate class on $Z=X^{2N(g)}$
for all primes $l \ne \fchar(k)$. This ends the proof.
\end{proof}

\begin{proof}[Proof of Theorem \ref{semigroupTate}]
In light of arguments of Subsection \ref{etale} combined with Remark \ref{product TateCl}, it suffices to check that each $\ell$-adic Tate form of any even degree $2m$ on  $X^n$
can be presented as a linear combination of exterior products of $\ell$-adic Tate forms of degree at most $H(g)$.  Let us prove it, using induction by $m$.

The assertion is obviously true
for all $m \le H(g)/2$.  Suppose that $m> H(g)/2$.  First, notice that
$$R_{X^n}=R_X \ \forall n.$$
Applying Theorem \ref{semigroupADM}, we conclude that the conditions of Proposition \ref{lastEff} are fulfilled for
$$K=\Q_{\ell},  V=V_{\ell}(X^n),  A=\Fr_{X^n}: V_{\ell}(X^n) \to V_{\ell}(X^n), $$
$$  \mathrm{spec}(A)=R_{X^n}=R_X,  \  h=H(g)/2.$$
Applying Proposition \ref{lastEff}, we conclude that each $\ell$-adic Tate form of degree $2m$ on  $X^n$ can be presented
as a linear combination of wedge products 
$$\psi_{m-j}\wedge \phi_j \ (j=1, \dots, H(g)/2))$$
 where
$\psi_{m-j}$ is an $\ell$-adic Tate form of degree $2(m-j)$ on $X^n$  and $\phi_j$ is an $\ell$-adic Tate form of degree $2j\le H(g)$ on $X^n$. Applying the induction assumption to all
$\psi_{m-j}$'s, we conclude that each $\ell$-adic Tate form of  degree $2m$ on  $X^n$
can be presented as a linear combination of exterior products of Tate forms of degree at most $H(g)$.
This ends the proof.
\end{proof}

\end{document}